\documentclass[12pt,a4paper]{article}
\usepackage{amsmath,amsthm,amsfonts,amssymb,bbm,color}
\usepackage{graphicx,psfrag,subfigure,url}
\usepackage{cite}
\usepackage{hyperref}
\usepackage[english]{babel}
\usepackage{enumitem}
\newcommand{\Id}{\mathbbm{1}}
\newcommand{\LL}{\mathcal{L}}
\newcommand{\Or}{\mathcal{O}}

\newcommand{\E}{\mathbbm{E}}
\newcommand{\e}{\varepsilon}
\newcommand{\I}{{\rm i}}

\newcommand{\R}{\mathbb{R}}
\newcommand{\N}{\mathbb{N}}
\newcommand{\Z}{\mathbb{Z}}

\renewcommand{\Re}{\operatorname{Re}}

\numberwithin{equation}{section}

\DeclareMathOperator*{\Tr}{tr}

\DeclareMathOperator*{\Cov}{Cov}
\DeclareMathOperator{\Ai}{Ai}

\DeclareMathOperator{\Pb}{\mathbbm{P}}

\newtheorem{prop}{Proposition}[section]
\newtheorem{thm}[prop]{Theorem}
\newtheorem{lem}[prop]{Lemma}

\newtheorem{cor}[prop]{Corollary}

\newtheorem{rem}[prop]{Remark}

\title{On the exponent governing the correlation decay of the Airy$_1$ process}

\author{Riddhipratim Basu\thanks{International Centre for Theoretical Sciences, Tata Institute of Fundamental
Research, Bangalore, India; E-mail:~\texttt{rbasu@icts.res.in}} \and Ofer Busani \thanks{Institute for Applied Mathematics, University of Bonn, Endenicher Allee 60,\newline 53115 Bonn, Germany; E-mail:~\texttt{busani@iam.uni-bonn.de}} \and Patrik L. Ferrari\thanks{Institute for Applied Mathematics, University of Bonn, Endenicher Allee 60,\newline 53115 Bonn, Germany; E-mail:~\texttt{ferrari@uni-bonn.de}}}

\date{June 15, 2022}

\begin{document}
	\sloppy
	\maketitle
	
\begin{abstract}
We study the decay of the covariance of the Airy$_1$ process, ${\cal A}_1$, a stationary stochastic process on $\R$ that arises as a universal scaling limit in the Kardar-Parisi-Zhang (KPZ) universality class. We show that the decay is super-exponential and determine the leading order term in the exponent by showing that $\Cov({\cal A}_1(0),{\cal A}_1(u))= e^{-(\frac{4}{3}+o(1))u^3}$ as $u\to\infty$. The proof employs a combination of probabilistic techniques and integrable probability estimates. The upper bound uses the connection of ${\cal A}_1$ to planar exponential last passage percolation and several new results on the geometry of point-to-line geodesics in the latter model which are of independent interest; while the lower bound is primarily analytic, using the Fredholm determinant expressions for the two point function of the Airy$_1$ process together with the FKG inequality.
\end{abstract}

\section{Introduction and the main result}\label{sectIntro}
The one-dimensional Kardar-Parisi-Zhang (KPZ) universality class~\cite{KPZ86} of stochastic growth models has received a lot of attention in recent years, see e.g.~the surveys and lecture notes~\cite{FS10,Cor11,QS15,BG12,Qua11,Fer10b,Tak16,Zyg18}. Two of the most studied models in this class are the exponential/geometric last passage percolation (LPP) and the totally asymmetric simple exclusion process (TASEP). In both cases, one can define a height function $h(x,t)$, where $x$ stands for space (one-dimensional in our case) and $t$ for time.

At a large time $t$, under the $2/3-1/3$ scaling, one expects to see a non-trivial limit process. To illustrate it, consider the scaling around the origin
\begin{equation}
h_t^{\rm resc}(u)=\frac{h(u t^{2/3},t)-t h_{\rm ma}(u t^{-1/3})}{t^{1/3}}
\end{equation}
with $h_{\rm ma}(\xi)=\lim_{t\to\infty} t^{-1} h(\xi t)$ being the (deterministic) macroscopic limit shape.

The limit process depends on the geometry of the initial condition. One natural initial condition is the stationary one and the limit process in this case, called Airy$_{\rm stat}$, has been determined in~\cite{BFP09}. For non-random initial conditions, the two main cases are:
\begin{enumerate}
\item[(a)] curved limit shape $h_{\rm ma}$: one expects the weak limit $\lim_{t\to\infty} h^{\rm resc}_t(u)=a_1 {\cal A}_2(a_2 u)$, with ${\cal A}_2$ being the Airy$_2$ process~\cite{PS02} and $a_1,a_2$ are model-dependent parameters (see~\cite{PS02,Jo03b,BF07} for LPP and TASEP setting and~\cite{Dim20} for a non-determinantal case),
\item[(b)] flat limit shape $h_{\rm ma}$: one expects the weak limit $\lim_{t\to\infty} h^{\rm resc}_t(u)=a_1' {\cal A}_1(a_2' u)$, with ${\cal A}_1$ being the Airy$_1$ process~\cite{Sas05}, with again $a_1',a_2'$ model-dependent parameters (see~\cite{Sas05,BFPS06,BFP06}).
\end{enumerate}

As universal limit objects in the KPZ universality class, the $Airy_{\rm stat}$, as well as ${\cal A}_1$ and ${\cal A}_2$ (which also are stationary stochastic processes in $\R$) have attracted much attention. It is known that the one point marginal for ${\cal A}_2$ is the GUE Tracy-Widom distribution from random matrix theory~\cite{PS02}, whereas the one point marginal for ${\cal A}_1$ is a scalar multiple of the GOE Tracy-Widom distribution~\cite{Sas05,BFPS06}. The next fundamental question is naturally to understand the two point functions for these processes. Although there are explicit formulae available for the multi-point distributions, extracting asymptotics from these complicated formulae is non-trivial. Widom in~\cite{Wid03} (see also~\cite{aVm03} for a conditional result) proved that
\begin{equation}
{\rm Cov}({\cal A}_2(0),{\cal A}_2(u))=2 u^{-2}+\Or(u^{-4})\textrm{ as }u\to\infty.
\end{equation}
Although algebraically there are many similarities between the processes ${\cal A}_1$ and ${\cal A}_{2}$ (see the review~\cite{Fer07}), the method used in~\cite{Wid03} can not be directly applied to the case of the Airy$_1$ process, and the question of understanding the decay of correlations in the Airy$_1$ process had remained open until now.

A numerical study~\cite{BFP08} clearly showed that the decay of the covariance for the Airy$_1$ process is very different from that of Airy$_2$, in that it decays super-exponentially fast, i.e., $-\ln \Cov({\cal A}_1(0),{\cal A}_1(u)) \sim u^{\delta}$ for some $\delta>1$. Unfortunately, the numerical data of~\cite{BFP08} are coming from a Matlab program developed in~\cite{Born08} and uses the 10-digits machine precision. From the data it was not possible to conjecture the true value of $\delta$.

The reason behind the difference in the decay of the covariances of ${\cal A}_2$ and ${\cal A}_1$ can be explained as follows. In the curved limit shape situation, the space-time regions which essentially determine the values of $h(0,t)$ and $h(u t^{2/3},t)$ have an intersection whose size decays polynomially in $u$. In contrast, for the flat limit shape, except on a set whose probability goes to zero super-exponentially fast in $u$, these regions are disjoint.

The goal of this paper is to prove that the decay of covariance for the Airy$_1$ process is super-exponential with $\delta=3$. More precisely, we prove upper and lower bounds of the covariance where exponents have a matching leading order term. The following theorem is the main result of this paper.

\begin{thm}\label{ThmMain}
There exist constants $c,c'>0$ such that for $u>1$
\begin{equation}
e^{-c u \ln(u)} e^{-\frac43 u^3}\leq \Cov({\cal A}_1(0),{\cal A}_1(u)) \leq e^{c'u^2}e^{-\frac43 u^3}.
\end{equation}
\end{thm}
Clearly, the threshold $u>1$ above is arbitrary, and by changing the constants $c,c'$ we can get the same bounds for any $u$ bounded away from $0$.

The upper and the lower bounds in Theorem~\ref{ThmMain} are proved separately with very different arguments.

In Section~\ref{sectUpperBound} we prove the upper bound; see Corollary~\ref{corUB}. For this purpose we consider the point-to-line exponential last passage percolation (LPP) which is known to converge to the Airy$_1$ process under an appropriate scaling limit. Corollary~\ref{corUB} is an immediate consequence of Theorem~\ref{thmUpperBound} which proves the corresponding decorrelation statement in the LPP setting. The strategy for the upper bound follows the intuition that the decorrelation comes from the fact that the point-to-line geodesics for two initial points far from each other use mostly disjoint sets of random variables. To make this precise, we prove and use results controlling the transversal fluctuations of the point-to-line geodesics and their coalescence probabilities (see Theorem~\ref{ThmUB6Bis} and Theorem~\ref{ThmCrossings}) that are of independent interest. The proof is mainly based on probabilistic arguments, but uses one point moderate deviation estimates for the point-to-point and point-to-line exponential LPP with optimal exponents. Such results have previously been proved in~\cite{LR10}, but an estimate with the correct leading order term in the upper tail exponent is required for our purposes. These are obtained in Lemma~\ref{lemUB2} and Lemma~\ref{lemUB2B} by using asymptotic analysis.

In Section~\ref{SectLowerBound} we prove the lower bound; see Theorem~\ref{ThmLowerBound}. For the lower bound we start with Hoeffding's covariance formula, which says that the covariance of two random variables is given by the double integral of the difference between their joint distribution and the product of the two marginals; see \eqref{covid}. The joint distribution of the Airy$_1$ process is given in terms of a Fredholm determinant (see \eqref{Fr}) and the proof uses analytic arguments to obtain precise estimates for these Fredholm determinants. A crucial probabilistic step here, however, is the use of the FKG inequality applied in the LPP setting, which, upon taking an appropriate scaling limit yields that the aforementioned integrand is always non-negative; see Lemma~\ref{lemFKG}. This allows one to lower bound the covariance by estimating the integrand only on a suitably chosen compact set, which nonetheless leads to a lower bound with the correct value of the leading order exponent.

We finish this section with a brief discussion of some related works. Studying the decay of correlations in exponential LPP has recently received considerable attention. Following the conjectures in the partly rigorous work~\cite{FS16}, the decay of correlations in the time direction has been studied for the stationary and droplet initial conditions in~\cite{FO18}, where precise first order asymptotics were obtained (see also~\cite{FO22,BKLD22} for works on the half-space geometry). Similar, but less precise, estimates for the droplet and flat initial conditions were obtained in~\cite{BG18, BGZ19}. All these works also rely on understanding the localization and geometry of geodesics in LPP, some of those results are also useful for us. The lower bounds in~\cite{BG18, BGZ19} also use the FKG inequality in the LPP setting and provide bounds valid in the pre-limit. One might expect that similar arguments can lead to a bound similar, but quantitatively weaker, to Theorem~\ref{thmUpperBound} valid in the LPP setting.

\paragraph{Acknowledgements.} The work of O. Busani and P.L. Ferrari was partly funded by the Deutsche Forschungsgemeinschaft (DFG, German Research Foundation) under Germany’s Excellence Strategy - GZ 2047/1, projekt-id 390685813. P.L. Ferrari was also supported by the Deutsche Forschungsgemeinschaft (DFG, German Research Foundation) - Projektnummer 211504053 - SFB 1060. R. Basu is partially supported by a Ramanujan Fellowship (SB/S2/RJN-097/2017) and a MATRICS grant (MTR/2021/000093) from SERB, Govt. of India, DAE project no. RTI4001 via ICTS, and the Infosys Foundation via the Infosys-Chandrasekharan Virtual Centre for Random Geometry of TIFR.

\newpage
\section{Upper Bound}\label{sectUpperBound}
In this section we prove the upper bound of Theorem~\ref{ThmMain}.

\subsection{Last passage percolation setting}
We consider exponential last passage percolation (LPP) on $\Z^2$. Let $\omega_{i,j}\sim\exp(1)$, $i,j\in\Z$, be independent exponentially distributed random variables with parameter $1$. For points $u,v\in \Z^2$ with $u\prec v$, i.e., $u_1\leq v_1$ and $u_2\leq v_2$, we denote the passage time between the points $u$ and $v$ by
\begin{equation}
L_{u,v}=\max_{\pi:u\to v} \sum_{(i,j)\in\pi} \omega_{i,j},
\end{equation}
where the maximum is taken over all up-right paths from $u$ to $v$ in $\Z^2$. Denote by $\Gamma_{u,v}$ the geodesic from $u$ to $v$, that is, the path $\pi$ maximizing the above sum. Furthermore, let $\LL_n=\{(x,y)\in\Z^2\,|\, x+y=n\}$ and denote by
\begin{equation}
L_{u,\LL_n}=\max_{\pi:u\to\LL_n} \sum_{(i,j)\in\pi}\omega_{i,j}
\end{equation}
the point-to-line last passage time, where the maximum is taken over all up-right paths going from $u$ to $\LL_n$.

Let us mention some known limiting results of exponential LPP. Let\footnote{We do not write explicitly the rounding to integer values, i.e., $(x,y)$ stands for $(\lfloor x\rfloor,\lfloor y\rfloor)$.}
\begin{equation}
I(u)=u(2N)^{2/3}(1,-1),\quad J(u)=(N,N)+u(2N)^{2/3}(1,-1)
\end{equation}
and define the rescaled LPP
\begin{equation}\label{eq1.4}
L_N^*(u)=\frac{L_{I(u),\LL_{2N}}-4N}{2^{4/3}N^{1/3}}, \quad L_N(u)=\frac{L_{I(u),(N,N)}-4N}{2^{4/3}N^{1/3}}.
\end{equation}
Then, by the result on TASEP with density $1/2$~\cite{Sas05,BFPS06} which can be transferred to LPP using slow decorrelation~\cite{CFP10b}, we know that
\begin{equation}\label{eqUB5}
\lim_{N\to\infty} L_N^*(u)= 2^{1/3}{\cal A}_1(2^{-2/3}u),
\end{equation}
where ${\cal A}_1$ is the Airy$_1$ process, in the sense of finite-dimensional distributions. Similarly, (see~\cite{Jo03b} for the geometric case and~\cite{BP07} for a two-parameter extension)
\begin{equation}
\label{eqa2}
\lim_{N\to\infty} L_N(u)= {\cal A}_2(u)-u^2,
\end{equation}
with ${\cal A}_2$ is the Airy$_2$ process~\cite{PS02}, where in~\cite{Jo03} the convergence is weak convergence on compact sets.
\begin{figure}[t!]
 \centering
 \includegraphics[height=5.5cm]{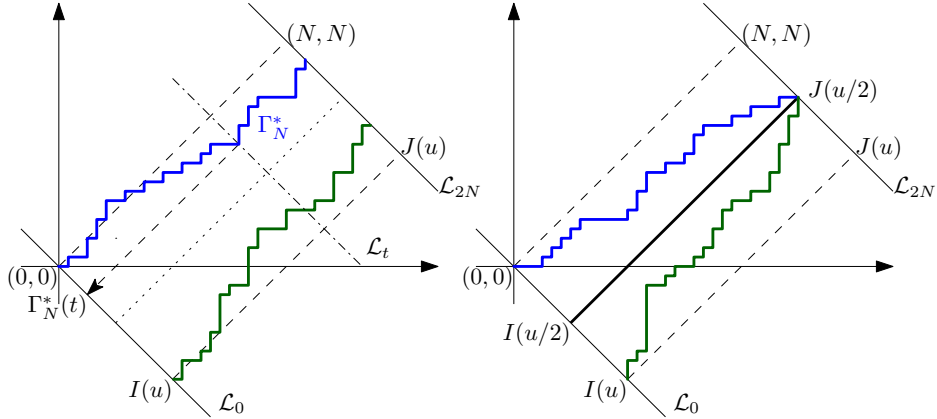}
\caption{
(Left): In the typical scenario, the geodesics associated with $L_N^*(0)$ (blue) and $L_N^*(u)$ (green) do not cross the straight line connecting $I(u/2)$ and $J(u/2)$. This suggests that $L_N^*(0)$ and $L_N^*(u)$ are almost independent. (Right) In the untypical case where the blue and green geodesics meet. In that case the geodesics will meet around the point $J(u/2)$. The probability of this event is the main contributor to the covariance of $L_N^*(0)$ and $L_N^*(u)$.
}
\label{FigGeometryUB}
\end{figure}

We also denote by $\Gamma_N$ (resp.\ $\Gamma^*_N$) the (almost surely unique) \emph{geodesic} attaining $L_{(0,0),(N,N)}$ (resp.\ $L_{(0,0),\LL_{2N}}$). For a directed path $\pi$, we denote by $\pi(t)=(x-y)/2$ where $(x,y)$ is the unique point (if it exits) where $\pi$ intersect $\LL_t$. The parameter $t$ will be thought of as \emph{time} and $\pi(t)$ will be the \emph{position} of the path at time $t$. For instance, if $\Gamma^*_N$ ends at $J(u)$, then $\Gamma^*_N(2N)=u(2N)^{2/3}$, see also Figure~\ref{FigGeometryUB}. We shall also denote by $L(\pi)$ the passage time of the path $\pi$, i.e., the sum of the weights on $\pi$.

The rescaled last passage times $L_N^*(u)$ and $L_N(u)$ have super-exponential upper and lower tails (see e.g.~Appendices of~\cite{FO18} for a collection of such results and references), which implies that the limit of their covariance is the covariance of their limit, i.e.,
\begin{equation}\label{eq2.7}
\lim_{N\to\infty} {\rm Cov}\left(L_N^*(u),L_N^*(0)\right)= 2^{2/3}{\rm Cov}\left({\cal A}_1(2^{-2/3}u),{\cal A}_1(0)\right).
\end{equation}
So if ${\rm Cov}\left(L_N^*(u),L_N^*(0)\right)\sim e^{-\beta u^3}$, then ${\rm Cov}\left({\cal A}_1(u),{\cal A}_1(0)\right)\sim e^{-4 \beta u^3}$.

\medskip

We first state the upper bound on ${\rm Cov}\left(L_N^*(u),L_N^*(0)\right)$ which is the main result in this section.

\begin{thm}\label{thmUpperBound}
For $N^{1/14}\gg u>1$,
\begin{equation}
{\rm Cov}\left(L_N^*(u),L_N^*(0)\right)\leq e^{cu^2}e^{- \frac13 u^3}
\end{equation}
for some $c>0$.
\end{thm}

The following corollary, proving the upper bound in Theorem~\ref{ThmMain}, is immediate from \eqref{eq2.7} and the above theorem.

\begin{cor}\label{corUB}
For $u>1$, we have
\begin{equation}
{\rm Cov}\left({\cal A}_1(u),{\cal A}_1(0)\right) \leq e^{cu^2}e^{-\frac43 u^3}
\end{equation}
for some $c>0$.
\end{cor}

Before proceeding further, let us explain the heuristic idea behind the proof of Theorem~\ref{thmUpperBound}. Let $T$ denote the straight line joining $I(u/2)$ and $J(u/2)$. Let $\tilde{L}_{N}(0)$ (resp.\ $\tilde{L}_{N}(u)$) denote the rescaled last passage time from $(0,0)$ to $\LL_{2N}$ (resp.\ from $I(u)$ to $\LL_{2N}$) in the LPP restricted to use the randomness only to the left (resp.\ to the right) of $T$. Since $\tilde{L}_{N}(0)$ and\ $\tilde{L}_{N}(u)$ depend on disjoint sets of vertex weights and hence are independent, one expects that the leading order behaviour of the covariance is given by the probability that $L_N^*(0)\neq \tilde{L}_{N}(0)$ and $L_N^*(u)\neq \tilde{L}_{N}(u)$ (parts of the sample space where only one of these two events hold can also contribute, but our arguments will show that these contributions are not of a higher order, see Figure~\ref{FigGeometryUB}). Now,
$$\Pb(L_N^*(0)\neq \tilde{L}_{N}(0))=\Pb(L_N^*(u)\neq \tilde{L}_{N}(u))=\Pb\Big(\sup_{0\leq t\leq 2N}\Gamma^*_N(t)\geq \tfrac{1}{2}u(2N)^{2/3}\Big)$$
and the probability of the last event is $\lesssim e^{-\frac{1}{6}u^3}$ by Theorem~\ref{ThmUB6Bis} below. The proof is completed by showing that the probability of the intersection of the two events has an upper bound which is of the same order (at the level of exponents) as their product. This final step is obtained by considering the two cases, one where the point-to-line geodesics do not intersect, and the second where they do. The first part is bounded using the BK inequality where the probability of the geodesics intersecting is upper bounded separately in Theorem~\ref{ThmCrossings}.

\subsection{Localization estimates of geodesics}
As explained above, a key step in the proof is to get precise estimates for the probability that the geodesic behaves atypically, i.e., it exits certain given regions. To this end, the main result of this subsection provides the following localization estimate for $\Gamma^*_N$ that is of independent interest.

\begin{thm}\label{ThmUB6Bis}
For $N^{1/14}\gg u>1$ we have
\begin{equation}
\Pb\Big(\sup_{0\leq t\leq 2N}\Gamma^*_N(t)\geq u(2N)^{2/3}\Big)\leq e^{cu^2}e^{-\frac{4}{3} u^3}
\end{equation}
for some constant $c>0$.
\end{thm}

Although we do not prove a matching lower bound, the constant $\frac{4}{3}$ is expected to be optimal; see the discussion following Lemma~\ref{lemUB3}. Transversal fluctuation estimates for geodesics in LPP are of substantial interest and have found many applications. This is a first optimal upper bound in this direction for point-to-line geodesic. For point-to-point geodesics, similar estimates (albeit with unspecified constants in front of the cubic exponent) are proved for Poissonian LPP~\cite{BSS14} and exponential LPP~\cite{BGZ19}, see also~\cite{HS20} for a lower bound. In fact, we shall need to use the following estimate from~\cite{BGZ19,BF20b}.

\begin{lem}[Proposition~C.9 of~\cite{BGZ19}]\label{lemUB5}
For $N^{1/3}\gg u>1$ we have
\begin{equation}
\Pb\Big(\sup_{0\leq t\leq 2N}\Gamma_N(t)\geq u(2N)^{2/3}\Big)\leq e^{-c u^3}
\end{equation}
for some constant $c>0$.
\end{lem}

For the proof of Theorem~\ref{ThmUB6Bis} as well as the results in the subsequent subsections of this section, we shall use as input some lower and upper tail estimates of various LPP, which are collected and, if needed, proved in Appendix~\ref{a:LPP}.

The first step is to prove the special case $t=2N$ of Theorem~\ref{ThmUB6Bis}; we get an estimate of the probability that $\Gamma^*_N$ ends in \mbox{${\cal D}_u=\cup_{v\geq u} J(v)$}, that is, $\Gamma^*_N(2N)\ge u(2N)^{2/3}$.

\begin{lem}\label{lemUB3}
For all $N^{1/9}\gg u>1$, we have
\begin{equation}\label{eqUBLocLine}
\Pb(\Gamma^*_N(2N)\geq u(2N)^{2/3})\leq e^{cu^2}e^{-\frac43u^3}
\end{equation}
for some $c>0$.
\end{lem}

Again, we do not prove a matching lower bound, but the constant $\frac{4}{3}$ should be optimal. Indeed, notice that by \eqref{eqa2}, one expects that $(2N)^{-2/3}\Gamma^{*}_{N}(2N)$ weakly converges to the almost surely unique maximizer ${\cal M}$ of ${\cal A}_{2}(u)-u^2$. The distribution of ${\cal M}$ has been studied in~\cite{Sch12,BKS12} whence it is known that $\Pb(|{\cal M}|\ge u)\sim e^{-\frac43u^3}$.

\begin{proof}[Proof of Lemma~\ref{lemUB3}]
Let $C_0$ be such that (using Lemma~\ref{lemUB1})
\begin{equation}\label{eq2.12}
\Pb(L_{(0,0),(N,N)}<4N- C_0u2^{4/3} N^{1/3})\le e^{-\frac{4}{3}u^3}.
\end{equation}
We have
\begin{equation}
\begin{aligned}
\Pb(\Gamma^*_N(2N)\leq u(2N)^{2/3})\geq &\, 1-\Pb(L_{(0,0),{\cal D}_u}\geq 4N-C_0u 2^{4/3}N^{1/3})\\
&-\Pb(L_{(0,0),(N,N)}<4N-C_0u 2^{4/3} N^{1/3}).
\end{aligned}
\end{equation}
Using our definition of $C_0$ and Lemma~\ref{lemUB2}, we get for $N^{1/3}\gg u\ge C_0+1$
\begin{equation}
\Pb(\Gamma^*_N(2N)\geq u(2N)^{2/3}) \leq C e^{-\frac43 (u^2-C_0u)^{3/2}}+ e^{-\frac{4}{3}u^3}\leq e^{cu^2} e^{-\frac{4}{3} u^3}.
\end{equation}
By adjusting the constant $c$ if necessary, we get the same conclusion for $u\in [1,C_0+1]$.
\end{proof}

Our next result is a similar localization estimate for the point-to-line geodesic $\Gamma^*_N$ at an intermediate time $t$.
\begin{lem}\label{lemUB3B}
Let $t=2\tau N$. Fix any $\theta\geq 1$ and take $\min\{\tau,1-\tau\}\geq u^{-\theta}$. Assume that $N^{1/(9+3\theta)}\gg u>\max\{1,\tfrac43\sqrt{\tau}C_0\}$ with $C_0$ as in \eqref{eq2.12}. Then there exists a constant $c>0$ such that
\begin{equation}
\label{eq:t1}
\Pb(\Gamma^*_N(t)\geq u(2N)^{2/3}) \leq e^{-\frac43 u^3 \tau^{-3/2}}e^{\frac12 u \tau^{-3/2}} e^{u^2(c+ 2 C_0 \tau^{-1})}.
\end{equation}
In particular, by choosing another constant $c'>0$,
\begin{equation}
\label{eq:t2}
\Pb(\Gamma^*_N(t)\geq u(2N)^{2/3})\leq e^{-\frac43 u^3+c' u^2}
\end{equation}
for all $N^{1/(9+3\theta)}\gg u>1$.
\end{lem}

Notice that \eqref{eq:t1} provides a stronger bound compared to \eqref{eq:t2} for small $\tau$ which is expected as the transversal fluctuation should grow with $\tau$. Indeed, for $\tau \ll 1$, one expects even stronger bounds, see Theorem~3 of~\cite{BSS19}. For the proof of Theorem~\ref{ThmUB6Bis}, \eqref{eq:t2} would suffice but we record the stronger estimate \eqref{eq:t1} as it would be used in the next subsection.

\begin{proof}[Proof of Lemma~\ref{lemUB3B}]
We start with a simple rough bound. Notice that since the geodesics are almost surely unique, by planarity they cannot cross each other multiple times. Consequently, if $\Gamma^*_N(2N)\leq A(2N)^{3}$ for some $A>0$, the geodesic $\Gamma^{*}_{N}$ lies to the left of the point-to-point geodesic from $I(A)$ to $J(A)$. Therefore, the maximal transversal fluctuation of the point-to line geodesic can be upper bounded by the sum of the fluctuation at the endpoint plus the maximal transversal fluctuation of a point-to-point geodesic. Similar arguments will be used multiple times in the sequel and will be referred to as the ordering of geodesics. It follows that
\begin{equation}\label{Oub5}
\begin{aligned}
\Pb(\Gamma^*_N(t)\geq C_1 u\tau^{-1/2} (2N)^{2/3})&\leq \Pb(\Gamma^*_N(2N)\geq \tfrac12 C_1 u\tau^{-1/2} (2N)^{2/3})\\
&+\Pb(\Gamma_N(t)\geq \tfrac12 C_1 u\tau^{-1/2} (2N)^{2/3}).
\end{aligned}
\end{equation}
Applying the bounds of Lemma~\ref{lemUB3} and Lemma~\ref{lemUB5} with the constant $C_1$ large enough, we get that
\begin{equation}\label{Oub5b}
\Pb(\Gamma^*_N(t)\geq C_1 u\tau^{-1/2} (2N)^{2/3})\leq e^{-\frac43 u^3 \tau^{-3/2}}.
\end{equation}

Next we need to bound the probability that $\Gamma^*_N(t)$ is in $[u(2N)^{2/3},C_1u\tau^{-1/2}(2N)^{2/3}]$. Define $K(v)=(t/2,t/2)+v(2N)^{2/3}(1,-1)$. Then, for any $S\in \R$,
\begin{multline}\label{eqB4}
\Pb(C_1 u\tau^{-1/2}(2N)^{2/3}>\Gamma^*_N(t)\geq u(2N)^{2/3})
\leq \Pb(L_{(0,0),\LL_{2N}}\leq S)\\
+\Pb\Big(\sup_{u\leq v\leq C_1 u\tau^{-1/2}} (L_{(0,0),K(v)}+\tilde L_{K(v),\LL_{2N}})>S\Big),
\end{multline}
where $\tilde L_{K(v),\LL_{2N}}=L_{K(v),\LL_{2N}}-\omega_{K(v)}$ is the LPP without the first point\footnote{The tail bounds in Corollary~\ref{a:LPP} clearly continue to hold even after removing the random variable at $K(v)$ which is ${\rm Exp}(1)$-distributed. The advantage is that in this way $L_{(0,0),K(v)}$ and $\tilde L_{K(v),\LL_{2N}}$ are independent random variables.}.
We set
\begin{equation}
S=4N-a u^2 2^{4/3} N^{1/3}.
\end{equation}
By Corollary~\ref{CorUB2}, setting $a=C_0 \tau^{-1/2}u^{-1}\ll N^{2/3}$ with $C_0$ as in \eqref{eq2.12}, we see that
\begin{equation}
\Pb(L_{(0,0),\LL_{2N}}\leq S)\leq e^{-\frac43 u^3 \tau^{-3/2}}.
\end{equation}
It remains to bound the last term in \eqref{eqB4}. We have
\begin{equation}\label{eqB6}
\begin{aligned}
&\Pb\Big(\sup_{u\leq v\leq C_1 u \tau^{-1/2}} (L_{(0,0),K(v)}+\tilde L_{K(v),\LL_{2N}})>S\Big)\\
&\leq \sum_{k=0}^{C_1 u/(\sqrt{\tau}\delta)-u-1} \hspace{-1em}\Pb\Big(\sup_{u+k \delta\leq v\leq u+(k+1)\delta} (L_{(0,0),K(v)}+\tilde L_{K(v),\LL_{2N}})>S\Big)\\
&\leq \sum_{k=0}^{C_1 u/(\sqrt{\tau}\delta)-u-1} \hspace{-1em}\Pb\Big(\sup_{v\geq u+k\delta} L_{(0,0),K(v)}+\sup_{u+k\delta\leq v\leq u+(k+1)\delta}\tilde L_{K(v),\LL_{2N}}>S\Big),
\end{aligned}
\end{equation}
where $\delta>0$ will be chosen later. Note that the two supremums above are independent. Denoting $X_k=\sup_{v\geq u+k\delta} L_{(0,0),K(v)}$ and $Y=\sup_{-\delta/2\leq v\leq \delta/2}\tilde L_{K(v),\LL_{2N}}$, note also that $\sup_{u+k\delta\leq v\leq u+(k+1)\delta}\tilde L_{K(v),\LL_{2N}}$ has the same law as $Y$ for each $k$. Hence,
\begin{equation}\label{eqB7}
\eqref{eqB6} =\sum_{k=0}^{C_1 u/(\sqrt{\tau}\delta)-u-1} \Pb(X_k+Y>S) \leq \frac{C_1 u}{\sqrt{\tau}\delta} \Pb(X_0+Y>S)
\end{equation}
since $X_k\leq X_0$ for any $k\geq 0$.

As $X_0$ and $Y$ are independent, it is expected that the leading term should behave as
\begin{equation}\label{App9}
\Pb(X_0>S^*) \Pb(Y>S-S^*)
\end{equation}
with $S^*$ chosen such that \eqref{App9} is maximal. Since we want to minimize over a finite number of points (not going to infinity as $N$ does), we instead look for the maximum of
\begin{equation}\label{App9B}
\Pb(X_0>S^*) \Pb(Y>S-S^*-\eta 2^{4/3} N^{1/3})
\end{equation}
for some small positive discretization step $\eta$; see Figure~\ref{FigDiscretization}. The natural scale of the fluctuation of $X_0$ is $\tau^{1/3} 2^{4/3}N^{1/3}$ and the one for $Y$ is $(1-\tau)^{1/3}2^{4/3}N^{1/3}$, so we choose $\eta$ at most $\tfrac14\min\{(1-\tau)^{1/3},\tau^{1/3}\}$.
\begin{figure}[t!]
 \centering
 \includegraphics[height=6cm]{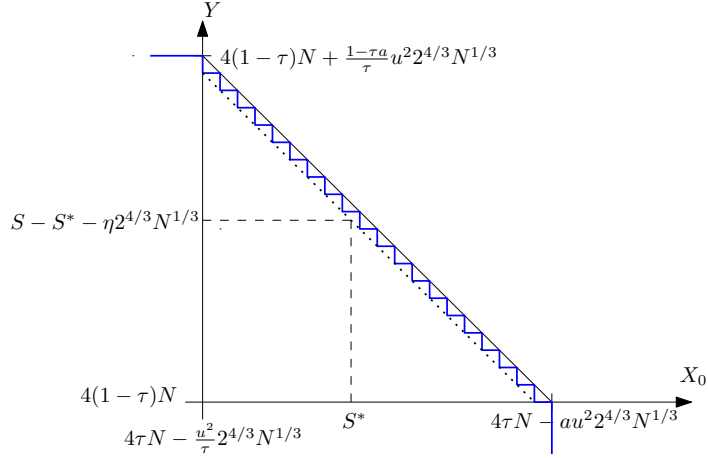}
\caption{The probability $\Pb(X_0+Y>S)$ is smaller than the sum of the probabilities $\Pb(X_0\geq A_1, Y\geq A_2)$ with $(A_1,A_2)$ being the discretized points on $A_1+A_2=S-\eta 2^{4/3}N^{1/3}$.}
\label{FigDiscretization}
\end{figure}

We want to discretize the interval $[4\tau N-\frac{u^2}{\tau}2^{4/3}N^{1/3},4\tau N-a u^2 2^{4/3}N^{1/3}]$ into pieces of size $\eta 2^{4/3} N^{1/3}$. The interval shall be non-empty, which can be ensured if $u$ is not too small. For that purpose let us assume that $u>\max\{1,\frac43 \sqrt{\tau}C_0\}$. Define the number of discretized points and then $\eta$ by
\begin{equation}
M=\left\lceil u^2(\tau^{-1}-a)\frac{4}{\min\{(1-\tau)^{1/3},\tau^{1/3}\}}\right\rceil,
\quad \eta=u^2(\tau^{-1}-a)\frac{1}{M}.
\end{equation}
Our assumption on $u$ ensures that $\frac1\tau-a\geq \frac{1}{4\tau}$, or equivalently $\tau a\leq \frac34$, as well as $\eta/u^2\leq \frac14 2^{-1/3}$. Therefore $\tau a+\eta/u^2<1$.

Let $S_1=4\tau N-\frac{u^2}{\tau}2^{4/3}N^{1/3}$ and $S_2=S-S_1=4(1-\tau)N+\frac{1-\tau a}{\tau}u^2 2^{4/3}N^{1/3}$, $S_3=4\tau N-a u^22^{4/3}N^{1/3}$. Let $S^*$ be the maximizer of \eqref{App9B}. Then
\begin{equation}\label{eqB17}
\begin{aligned}
\eqref{eqB7} &\leq \frac{C_1 u}{\sqrt{\tau}\delta} \sum_{k=0}^{M-1}\Pb\left(X_0>S_1+k\eta 2^{4/3}N^{1/3},Y>S_2- (k+1)\eta2^{4/3}N^{1/3}\right)\\
&+\frac{C_1 u}{\sqrt{\tau}\delta} \Pb\left(X_0>S_3\right)+\frac{C_1 u}{\sqrt{\tau}\delta}\Pb\left(Y>S_2\right)\\
&\leq\frac{C_1 u}{\sqrt{\tau}\delta} M\Pb\left(X_0>S^*\right) \Pb\left(Y>S-S^*- \eta 2^{4/3} N^{1/3}\right)\\
&+\frac{C_1 u}{\sqrt{\tau}\delta} \Pb\left(X_0>S_3\right)+\frac{C_1 u}{\sqrt{\tau}\delta}\Pb\left(Y>S_2\right).
\end{aligned}
\end{equation}

Instead of finding $S^*$ that maximizes \eqref{App9B}, we look to find an upper bound of \eqref{App9B} by maximizing the product of the upper bounds of the individual terms. For this, we take $\delta=(1-\tau)^{2/3}$. Then by
rescaling Lemma~\ref{lemUB2}, for $s_2\ll (\tau N)^{2/9}$ and $u\ll (\tau N)^{1/9}$
\begin{equation}\label{eqB10}
\Pb\Big(X_0>4\tau N-\frac{u^2}{\tau} 2^{4/3} N^{1/3}+s_1 \tau^{1/3} 2^{4/3}N^{1/3}\Big)\leq C e^{-\frac43 s_1^{3/2}}
\end{equation}
and by rescaling Proposition~\ref{PropUB5}, for $s_2\ll (1-\tau)^{2/3}N^{2/3}$, we get
\begin{equation}\label{eqB9}
\Pb\Big(Y>4(1-\tau) N+s_2(1-\tau)^{1/3} 2^{4/3}N^{1/3}\Big)\leq C s_2 e^{-\frac43 s_2^{3/2}}.
\end{equation}
Thus we need to maximize the product of the terms in \eqref{eqB10} and \eqref{eqB9}. All the terms in the sum in \eqref{eqB17} corresponds to $s_1$ and $s_2$ being non-negative and also of $\Or(\frac{u^2}{\tau \min\{\tau^{1/3},(1-\tau)^{1/3}\}})$. Therefore we can ignore the polynomial pre-factor in \eqref{eqB9} and look to find $s_1,s_2\geq 0$ such that
\begin{equation}
-\frac{u^2}{\tau}+s_1 \tau^{1/3}+s_2(1-\tau)^{1/3}=-(a u^2+\eta),\quad s_1^{3/2}+s_2^{3/2}\textrm{ is minimal}.
\end{equation}

Denote $\tilde a=a+\eta /u^2$, which satisfied $\tau\tilde a<1$ by the above assumptions. Plugging in the value of $s_2$ as function of $s_1$ into $s_1^{3/2}+s_2^{3/2}$ and computing its minimum through the derivatives we get
\begin{equation}
s_1=\left(\tfrac{1}{\tau}-\tilde a\right) \tau^{2/3}u^2,\quad s_2= \left(\tfrac{1}{\tau}-\tilde a\right) (1-\tau)^{2/3} u^2,\quad s_1^{3/2}+s_2^{3/2}=u^3\left(\tfrac{1}{\tau}-\tilde a\right)^{3/2}.
\end{equation}
With this choice of $s_1$ and $s_2$, we write with a minor abuse of notation
\begin{equation}
\begin{aligned}
&S^*=4\tau N-\frac{u^2}{\tau} 2^{4/3} N^{1/3}+s_1\tau^{1/3}2^{4/3}N^{1/3},\\
&S-S^*-\eta 2^{4/3}N^{1/3}=4(1-\tau)N+s_2(1-\tau)^{1/3}2^{4/3}N^{1/3},
\end{aligned}
\end{equation}
and
\begin{equation}\label{eqB12}
\Pb\left(X_0>S^*\right) \Pb\left(Y>S-S^*-\eta 2^{4/3}N^{1/3}\right)\leq C u^{2+4\theta/3} e^{-\frac43 u^3 \left(\tfrac{1}{\tau}-\tilde a\right)^{3/2}}
\end{equation}
for some constant $C>0$, where we have used the a priori bound on $s_2$ together with the assumption on $\tau$ to get the polynomial pre-factor. For $\tau \tilde a <1$, $\left(\tau^{-1}-\tilde a\right)^{3/2}\geq \tau^{-3/2}(1-\tfrac32 \tau a-\frac38 u^{-2})$ so that we get
\begin{equation}\label{eq2.36}
\eqref{eqB12} \leq C u^{2+4\theta/3} e^{-\frac43 u^3 \tau^{-3/2}} e^{\frac12 u \tau^{-3/2}} e^{2 u^2 C_0 \tau^{-1}}.
\end{equation}
This is the $\tau$-dependent bound which is useful for small enough $\tau$, but the exponent is minimal when $\tau\to 1$.

Finally, notice that the bound on $\Pb(X_0>S_3)$ (resp.\ $\Pb(Y>S_2)$) corresponds to the bound with the value $s_2=0$ (resp.\ $s_1=0$), thus they are also smaller than \eqref{eqB12}. Thus we have shown that $\eqref{eqB17}\leq (M+2)\frac{C_1 u}{\sqrt{\tau}\delta} \times \eqref{eq2.36}$. The prefactor is only a polynomial in $u$ (using again the assumptions on $\tau$) and can be absorbed in the $u^2$ term in the exponent by adjusting the constant. Thus we have proved \eqref{eq:t1}, and \eqref{eq:t2} follows by observing that $\tau\to 1$ is the worst case. The condition $u\ll N^{1/(9+3\theta)}$ ensures that all the conditions on $s_1$, $s_2$, $u$ mentioned above are satisfied.
\end{proof}

We can now prove Theorem~\ref{ThmUB6Bis}.
\begin{proof}[Proof of Theorem~\ref{ThmUB6Bis}]
We shall prove the result for $u$ sufficiently large, the result for all $u>1$ shall follow by adjusting the constant $c$. Let us set $\e=\delta u^{-3/2}$ for some $\delta>0$ to be chosen later ($\delta$ will be small but fixed and in particular will not depend on $u$ or $N$). Without loss of generality let us also assume that $\e N$ and $1/\e$ are both integers. Let us define the sequence of points
\begin{equation}
v_0=I(u-1), \cdots, v_{j}=I(u-1)+(j\e N, j\e N), \cdots, v_{\e^{-1}}=J(u-1).
\end{equation}
Let $A_{j}$ denote the event that $\Gamma^*_{N}(2j \e N)\ge (u-1)(2N)^{2/3}$ for $j=1,2, \ldots, \e^{-1}$, and let $B_{j}$ denote the event that $\sup _{t}\Gamma_{v_{j-1},v_{j}} \ge u (2N)^{2/3}$ for $j=1,2,\ldots, \e^{-1}$ where $\Gamma_{v_{j-1},v_{j}}$ denotes the geodesic from $v_{j-1}$ to $v_{j}$. By ordering of geodesics, it follows that on
\begin{equation}
\Big(\bigcap _{j} A_j^{c}\Big) \cap \Big(\bigcap _{j} B_j^{c}\Big)
\end{equation}
one has $\sup_{0\leq t\leq 2N} \Gamma^*_{N}(t) \le u(2N)^{2/3}$; see Figure~\ref{FigLocalization}.
\begin{figure}[t!]
 \centering
 \includegraphics[height=5cm]{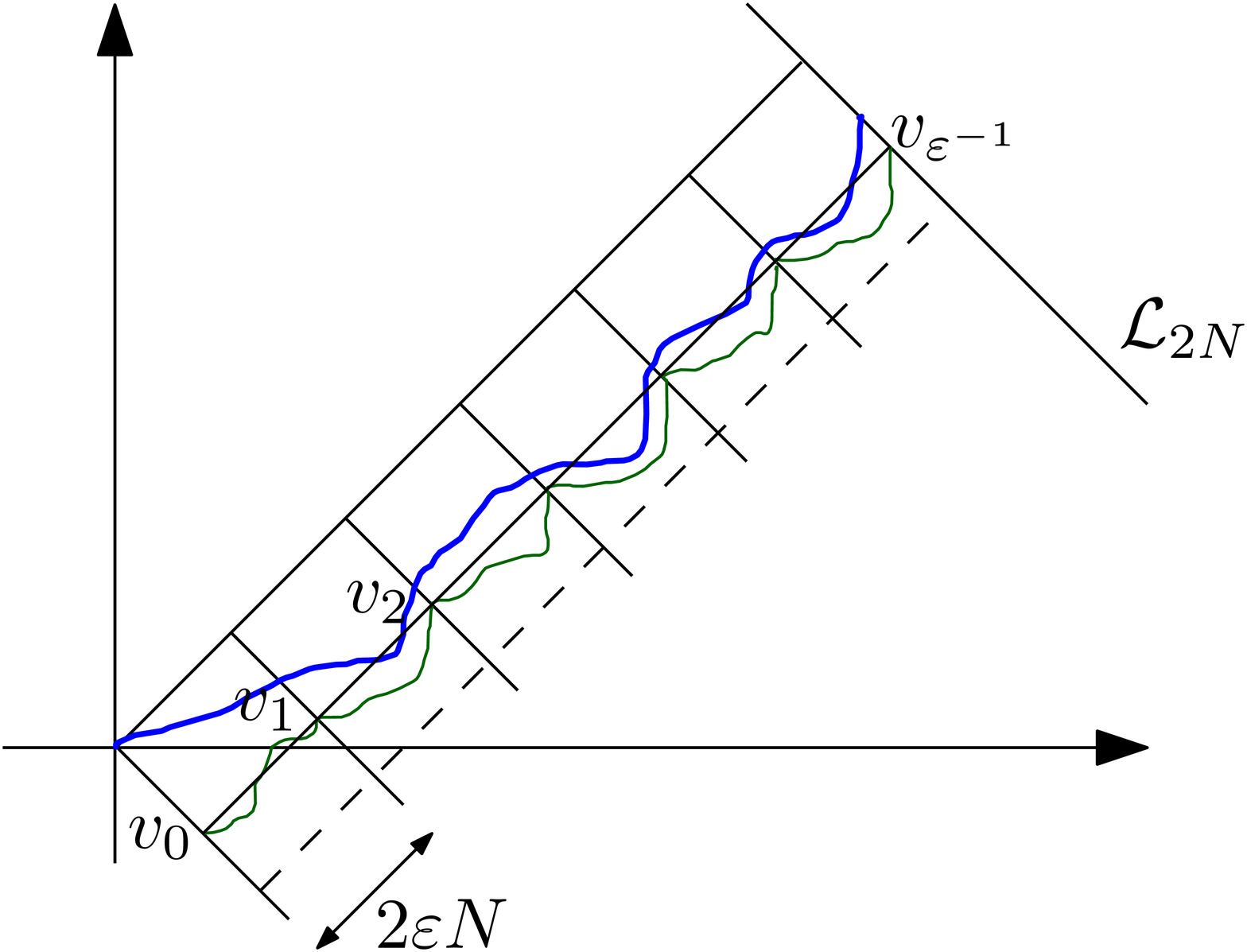}
\caption{The thick blue line is the geodesic $\Gamma^*_N$. It passes to the left of the $v_1,v_2,\ldots,v_{\e^{-1}}$. The green thin lines are the point-to-point geodesics from $v_j$ to $v_{j+1}$, $j=0,\ldots,\e^{-1}-1$, which stay to the left of the dashed line joining $I(u)$ and $J(u)$.}
\label{FigLocalization}
\end{figure}
Hence,
\begin{equation}
\Pb\Big(\sup_{0\leq t\leq 2N}\Gamma^*_N(t)\geq u(2N)^{2/3}\Big)\leq \sum_{j} \Pb(A_j)+ \sum_{j}\Pb(B_j).
\end{equation}
It follows from Lemma~\ref{lemUB3B} that for each $j$,
\begin{equation}
\Pb(A_{j}) \leq e^{c(u-1)^2} e^{-\frac{4}{3}(u-1)^3}\leq e^{c' u^2} e^{-\frac43 u^3}
\end{equation}
for some new constant $c'>0$.
Notice now that
\begin{equation}
\Pb(B_{j})=\Pb\Big(\sup_{0\leq t\leq 2\e N}\Gamma_{\e N}(t) \ge \delta^{-2/3}u(2\e N)^{2/3}\Big).
\end{equation}
We now use $u\ll N^{2/3}$ (which implies $u\ll (\e N)^{1/3}$) and choose $\delta$ sufficiently small such that by Lemma~\ref{lemUB5} we have
\begin{equation}
\Pb(B_{j}) \le e^{-\frac{4}{3}u^3}.
\end{equation}
With this choice it follows that
\begin{equation}
\Pb\Big(\sup_{0\leq t\leq 2N}\Gamma^*_N(t)\geq u(2N)^{2/3}\Big)\le 2\e^{-1}e^{c u^2}e^{-\frac{4}{3}u^3}.
\end{equation}
Since $\e^{-1}=\Or(u^{3/2})$, the result follows by adjusting the value of $c$.
\end{proof}

\subsection{Coalescence probability}
Consider the point-to-line geodesics $\Gamma^*_{N}$ and $\tilde{\Gamma}^{*}_{N}$ to $\LL_{2N}$ from $(0,0)$ and $I(u)$ respectively. Owing to the almost sure uniqueness of geodesics, if $\Gamma^*_{N}$ and $\tilde{\Gamma}^{*}_{N}$ meet, they \emph{coalesce} almost surely. Coalescence of geodesics is an important phenomenon in random growth models including first and last passage percolation and has attracted a lot of attention. For exponential LPP, for point-to-point geodesics started at distinct points and ending at a common far away point (or semi-infinite geodesics going in the same direction) tail estimates for distance to coalescence has been obtained; see~\cite{BSS19, Pim16, Zha20} for more on this. For point-to-line geodesics started at initial points that are far (on-scale), one expects the probability of coalescence to be small. Our next result proves an upper bound to this effect and is of independent interest.

\begin{thm}\label{ThmCrossings}
 In the above set-up, for $N^{1/14}\gg u>1$
\begin{equation}
\Pb(\Gamma^*_N\cap \tilde{\Gamma}^{*}_N\neq \emptyset)\leq e^{-\frac13 u^3+c u^2}
\end{equation}
for some constant $c>0$.
\end{thm}
The rest of this section deals with the proof of Theorem~\ref{ThmCrossings}. We divide it into several smaller results. As always, we shall assume without loss of generality that $u$ is sufficiently large, extending the results to all $u>1$ is achieved by adjusting constants.

First of all, due to Theorem~\ref{ThmUB6Bis}, the probability that the two geodesics in the statement of Theorem~\ref{ThmCrossings} meet outside the rectangle $\mathcal{R}(u)$ with corners $(0,0)$, $I(u)$, $J(u)$, $(N,N)$ is smaller than the estimate we want to prove. Thus we can restrict to bounding the probability that the two geodesics intersect in $\mathcal{R}(u)$; a stronger result is proved in Lemma~\ref{PropNotCrossingFar} below. As the number of points in $\mathcal{R}(u)$ where the geodesics $\tilde{\Gamma}^{*}_{N}$  and $\Gamma^*_N$ could meet is $\Or(N^{5/3})$, we need to discretize space. We therefore divide $\mathcal{R}(u)$ into a grid of size $\e N\times (2 \e N)^{2/3}$, where $\e$ will be taken small enough (but not too small, namely of order $u^{-2}$); see Figure~\ref{FigGrid}.
\begin{figure}[t!]
 \centering
 \includegraphics[height=6cm]{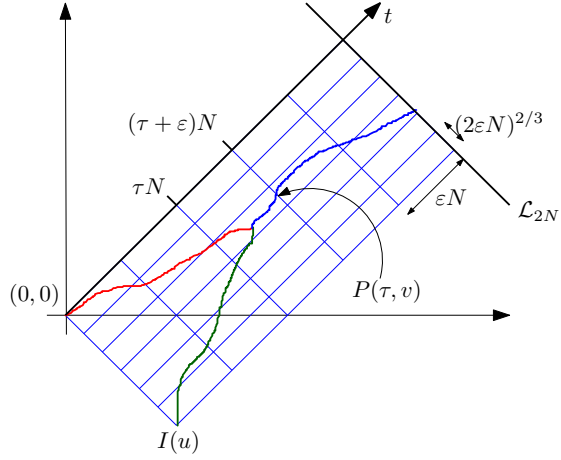}
\caption{Illustration of the grid as discretization of space-time. In the space direction the length is $(2\e N)^{2/3}$ and in the time direction $\e N$.}
\label{FigGrid}
\end{figure}

For $\tau$ an integer multiple of $\varepsilon$, let $A(\tau,v)$ be the event that the first intersection of $\Gamma^*_N$ and $\tilde{\Gamma}^{*}_N$ occurs at time $t \in (2\tau N,2\tau N+2\e N]$, and they then cross the anti-diagonal grid segment (of length $(2 \e N)^{2/3}$) at time $2\tau N+2\e N$ with mid-point given by
\begin{equation}
P(\tau,v)=(\tau N+\e N+v(2N)^{2/3},\tau N+\e N-v (2N)^{2/3}).
\end{equation}
Notice that, the number of choices of $\tau$ and $v$ is $\Or (\varepsilon^{-5/3})$, which by our choice of $\e$ is at most a polynomial of $u$.
Thus we need to prove the $\Pb (A(\tau,v))$ is at most $e^{-\frac13 u^3+c u^2}$ for any $\tau,v$. The proof of Theorem~\ref{ThmCrossings} is completed by taking a union bound.

Our first rough estimate deals with values of $v$ which are close to $0$ or $u$ and also small values of $\tau$. The basic idea is that in these cases the probability bounds coming from considering the transversal fluctuation of a single geodesic is sufficient.

\begin{lem}\label{PropNotCrossingFar}
Let $N^{1/14}\gg u >1$. For any $v$ satisfying $\min\{v,u-v\}\leq u (1-2^{-2/3})-\e^{2/3}$ and for any $\tau\leq 2^{-2/3}-\e$,
\begin{equation}
\Pb(A(\tau,v))\leq e^{-\frac13 u^3+c u^2}
\end{equation}
for some constant $c>0$.
\end{lem}
\begin{proof}
For $u-v\leq u(1-2^{-2/3})-\e^{2/3}$, we have
\begin{equation}
\Pb(A(\tau,v))\leq \Pb\Big(\sup_t \Gamma^*_N(t)\geq 2^{-2/3} u (2N)^{2/3}\Big)\leq e^{-\frac13 u^3+c u^2},
\end{equation}
where the last inequality follows from Theorem~\ref{ThmUB6Bis}. The same argument gives the desired result for $v\leq (1-2^{-2/3})u$ by considering the transversal fluctuation of the geodesic $\tilde{\Gamma}_{N}^*$.

Next, notice that $A(\tau,v)$ implies that either $\Gamma^*_N(2\tau N+2\e N)\geq \frac12 u (2N)^{2/3}$ or $\tilde{\Gamma}^{*}_{N}(2\tau N+2\e N)\leq \frac12 u (2N)^{2/3}$, since after meeting they follow the same path. For $\tau+\varepsilon \le 2^{-2/3}$, by Lemma~\ref{lemUB3B} (use the first inequality with $\tau\mapsto \tau+\e\leq 2^{-2/3}$) each of these events have probability bounded by
$e^{-\frac13 u^3+c u^2}$ for some constant $c>0$, completing the proof.
\end{proof}

We now proceed towards dealing with the remaining case. Define the segment
\begin{equation}
{\cal S}_v=\{(\tau N+k,\tau N-k)| (v-1)(2N)^{2/3}\leq k\leq (v+1)(2N)^{2/3}\}.
\end{equation}
Let $C_2$ be large enough such that
\begin{equation}\label{eqB27}
\Pb(L_{(0,0),\LL_{2N}}\leq 4N-C_2 u 2^{4/3}N^{1/3})\leq e^{-\frac13 u^3}.
\end{equation}
For a path $\gamma$, recall that $L(\gamma)$ denotes the passage time of that path. Define the event
\begin{multline}
B(\tau,v)=\{\exists\, \gamma_1,\gamma_2\,|\, \gamma_1:(0,0)\to {\cal S}_v,\gamma_2:I(u)\to {\cal S}_v,\gamma_1\cap\gamma_2=\emptyset\,
\textrm{ and }\\\min\{L(\gamma_1)+\tilde L_{{\cal S}_v,\LL_{2N}},L(\gamma_2)+\tilde L_{{\cal S}_v,\LL_{2N}}\}\geq 4N-C_2 u 2^{4/3}N^{1/3}\},
\end{multline}
where in $\tilde L$ we remove the first point. Then we have the following estimate.

\begin{figure}[t!]
 \centering
 \includegraphics[height=4.5cm]{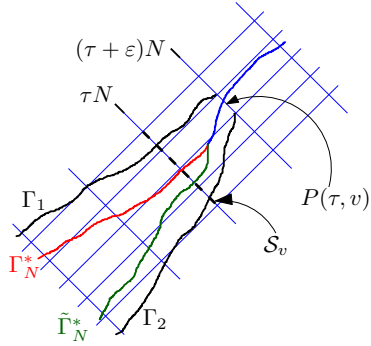}
\caption{Magnification of the \emph{local} geometry of geodesics used in the sandwitching of Lemma~\ref{AtauvBound}. The segment ${\cal S}_v$ is the dashed one. Notice that the lower line is not $\tau=0$.}
\label{FigSandwitching}
\end{figure}

\begin{lem}\label{AtauvBound}
Assume $(1-2^{-2/3})u-\e^{2/3}<v<2^{-2/3}u+\e^{2/3}$ and $2^{-2/3}-\e\leq \tau\leq 1-\varepsilon$ (notice that if the geodesics coalesce then $A(\tau,v)$ must hold for some $\tau\le 1-\e$, the case $\tau=1$ need not be considered).
For $N^{1/11}\gg u>1$, there exists a $\delta>0$ small enough (not depending on $u$ and $N$) such that with $\e=\delta u^{-2}$,
\begin{equation}
\Pb(A(\tau,v))\leq \Pb(B(\tau,v))+4 C e^{-\frac13 u^3}.
\end{equation}
for some constant $C>0$.
\end{lem}
\begin{proof}We prove it for $u>4(1+C_2)$ with $C_2$ as in \eqref{eqB27}. Then by adjusting the constant $C$ it is true also for $u>1$. Denote by $\Gamma_1$ and $\Gamma_2$ the geodesics from $(0,0)$ to $P(\tau,v-\frac12\e^{2/3})$ and from $I(u)$ to $P(\tau,v+\frac12\e^{2/3})$ respectively. These two points are the end point of the grid interval whose midpoint is $P(\tau,v)$; see also Figure~\ref{FigSandwitching}. Define the events
\begin{equation}
\begin{aligned}
B_1&=\{\Gamma_1(2\tau N)\leq (v-1)(2N)^{2/3}\},\\
B_2&=\{\Gamma_2(2\tau N)\geq (v+1)(2N)^{2/3}\},\\
B_3&=\{L_{(0,0),\LL_{2N}}\leq 4N-C_2 u 2^{4/3}N^{1/3}\},\\
B_4&=\{L_{I(u),\LL_{2N}}\leq 4N-C_2 u 2^{4/3}N^{1/3}\}.
\end{aligned}
\end{equation}
Let us show that
\begin{equation}\label{eqB32}
A(\tau,v)\subseteq B(\tau,v)\cup B_1\cup B_2 \cup B_3 \cup B_4.
\end{equation}
Observe that
\begin{equation}
A(\tau,v) \subseteq B_1\cup B_3\cup B_3 \cup B_4 \cup (A(\tau,v)\cap B_1^c\cap B_2^c\cap B_3^c\cap B_4^c)
\end{equation}
and the last event is included in $B(\tau,v)$. Indeed, on $A(\tau,v)\cap B_1^c\cap B_2^c$, the geodesics $\Gamma^*_N$ and $\tilde{\Gamma}^{*}_N$ must cross ${\cal S}_v$. Now, let $\gamma_1$ and $\gamma_2$ be the portions of $\Gamma^*_N$ and $\tilde{\Gamma}^{*}_N$ respectively before time $\tau N$. On $A(\tau,v)$, by definition, $\gamma_1$ and $\gamma_2$ must be disjoint. On $B_3^c\cap A(\tau,v)$ it holds that $L(\gamma_1)+L_{{\cal S}_v,\LL_{2N}}\geq L_{(0,0),\LL_{2N}}\geq 4N-C_2 u 2^{4/3}N^{1/3}$, and similar inequality holds on $B_4^c \cap A(\tau,v)$ replacing $\gamma_1$ by $\gamma_2$. Thus the event $B(\tau,v)$ is satisfied.

To complete the proof, we apply union bound to \eqref{eqB32} and bound the probabilities $\Pb(B_i)$. By the choice of $C_2$, $\Pb(B_3)$ and $\Pb(B_4)$ are both bounded by $e^{-\frac13 u^3}$. Lemma~\ref{lemBoundGamma1} below shows that $\Pb(B_1)\leq C e^{-\frac13 u^3}$ for some $C>0$ and by symmetry $\Pb(B_2)\leq C e^{-\frac13 u^3}$ as well. This completes the proof.
\end{proof}

\begin{lem}\label{lemBoundGamma1}
Assume $(1-2^{-2/3})u-\e^{2/3}<v<2^{-2/3}u+\e^{2/3}$ and $2^{-2/3}-\e\leq \tau\leq 1$.
For $N^{1/14}\gg u >1$, there exists a $\delta>0$ small enough (not depending on $u$ and $N$) such that with $\e=\delta u^{-2}$,
\begin{equation}
\Pb(\Gamma_1(2\tau N)\leq (v-1) (2N)^{2/3})\leq C e^{-\frac13 u^3}
\end{equation}
for some $C>0$.
\end{lem}
\begin{proof}
$\Gamma_1$ is the geodesic from $(0,0)$ to $Q_1=P(\tau,v-\frac12\e^{2/3})$. We want to bound the probability that $\Gamma_1(2\tau N)\leq (v-1)(2N)^{2/3}$.

For $x_1,x_2,x_3>0$, define the events
\begin{equation}
\begin{aligned}
E_1&=\{L_{(0,0),Q_1}\leq 4(\tau+\e)N-\tfrac{v^2}{\tau+\e} 2^{4/3}N^{1/3}-x_1 2^{4/3}N^{1/3}\},\\
E_2&=\{L_{(0,0),\LL_{2\tau N}}\geq 4\tau N+x_2 2^{4/3}N^{1/3}\},\\
E_3&=\Big\{\sup_{w\leq (v-1)(2N)^{2/3}}L_{(\tau N+w,\tau N-w),Q_1}\geq 4\e N\\ &\hspace{11em}-\tfrac{1}{\e}(1-\tfrac12 \e^{2/3})^22^{4/3}N^{1/3}+x_3 2^{4/3} N^{1/3}\Big\}.
\end{aligned}
\end{equation}
By a first order approximation, we have $L_{(0,0),Q_1}\simeq 4(\tau+\e)N-\frac{v^2}{\tau+\e}2^{4/3}N^{1/3}+\Or(v\e^{2/3}N^{1/3})$. So, by Lemma~\ref{lemUB1}, we have $\Pb(E_1)\leq C e^{-c x_1^3}$ for $x_1$ of at least the order of $u$, and such that $x_1\ll N^{2/3}$.
Next, by Lemma~\ref{lemUB2B}, we have $\Pb(E_2)\leq C e^{-\frac43 x_2^{3/2}\tau^{-1/2}}$ for $x_2\ll N^{2/9}$.
Finally, using Lemma~\ref{lemUB2} (with the variables $(N,u)$ in Lemma~\ref{lemUB2} replaced by $(\e N,\e^{-2/3}(1-\tfrac12 \e^{2/3}))$, we get $\Pb(E_3)\leq C e^{-\frac43 x_3^{3/2}\e^{-1/2}}$
provided
\begin{equation}\label{coe}
	x_3 \e^{-1/3}\ll(\e N)^{2/9} \text{\,\, and \,\, }  N^{-1/7}\ll\e.
\end{equation}
Under the condition
\begin{equation}\label{co}
	-\frac{v^2}{\tau+\e}-x_1\geq x_2-\frac{1}{\e}(1-\tfrac12 \e^{2/3})^2+x_3,
\end{equation}
 we have
\begin{equation}\label{eqA.54}
\Pb(\Gamma_1(2\tau N)\leq (v-1) (2N)^{2/3}) \leq \Pb(E_1)+\Pb(E_2)+\Pb(E_3).
\end{equation}
We assume already that $\e$ is small enough so that $\tau\geq 1/2$. First take $x_1=u/(3 c)$ so that $\Pb(E_1)\leq C e^{-\frac13 u^3}$. Then take $x_2=u^2$ which ensures $\Pb(E_2)\leq C e^{-\frac13 u^3}$ as well. Finally we take $x_3=u^2 \e^{1/3}$ that gives $\Pb(E_3)\leq C e^{-\frac13 u^3}$.  To satisfy the condition  \eqref{co}, it is enough to take $\e=\delta u^{-2}$ with $\delta$ small enough (independent of $u$). Finally, note that  \eqref{coe} implies that $u\ll N^{\tfrac{1}{14}}$.
\end{proof}

To complete the proof of Theorem~\ref{ThmCrossings} we need to obtain a bound on the event $B(\tau,v)$.
\begin{lem}\label{LemEventB}
Assume $2^{-2/3}u-\e^{2/3}<v<(1-2^{-2/3})u+\e^{2/3}$ and $2^{-2/3}-\e\leq \tau\leq 1-\varepsilon$.
For $N^{1/9}\gg u>4(1+C_2)$, there exists a constant $\delta>0$ small enough such that with $\e=\delta u^{-2}$,
\begin{equation}
\Pb(B(\tau,v))\leq e^{-\frac13 u^3+c u^2}
\end{equation}
for some constant $c>0$ independent of $\tau,v$. The constant $C_2$ is as in \eqref{eqB27}.
\end{lem}
\begin{proof}

For any $s_1,s_2\in\R_+\cup \{-\infty\}$ we define $D(s_1,s_2)$ to be the event that there exist disjoint paths $\gamma_1$ and $\gamma_2$ as in the definition of $B(\tau,v)$, such that
\begin{equation}
\begin{aligned}
L(\gamma_1)&\geq 4\tau N-\frac{(v-1)^2}{\tau}2^{4/3}N^{1/3}+ s_1 2^{4/3}N^{1/3},\\
L(\gamma_2)&\geq 4\tau N-\frac{(u-v-1)^2}{\tau}2^{4/3}N^{1/3}+ s_2 2^{4/3}N^{1/3}.
\end{aligned}
\end{equation}
For $s_3\in\R_+\cup \{-\infty\}$, we define the event
\begin{equation}
C(s_3)=\{\tilde L_{{\cal S}_v,\LL_{2N}}\geq 4(1-\tau)N+s_3 2^{4/3}N^{1/3}\}.
\end{equation}
Recall the constant $C_2$ from \eqref{eqB27}. Like in the proof of Lemma~\ref{lemUB3B} we do a discretization with a fixed width $0<\eta<1$ and thus we will not write all the details. The minor difference is that now we a couple of constraints:
\begin{equation}
s_1+s_3=\tfrac1\tau (v-1)^2-\eta-C_2 u,\quad s_2+s_3 = \tfrac1\tau (u-v-1)^2-\eta-C_2 u.
\end{equation}
In the discretization of Lemma~\ref{lemUB3B}, see \eqref{eqB17}, we separated explicitly two terms, which corresponds taking $S_2=-\infty$ and $S_3=-\infty$. Here we do the same, but instead of writing those terms separately, we consider subsets allowing positive numbers and $-\infty$. More precisely, define the set
\begin{equation}
\begin{aligned}
\Theta=\{&s_1,s_2,s_3 \in\R_+\cup \{-\infty\} |\, s_3\in \eta \Z, s_1\vee 0+s_3\vee 0 = \tfrac1\tau (v-1)^2-\eta-C_2 u\\ &\textrm{and }
s_2\vee 0+s_3\vee 0 = \tfrac1\tau (u-v-1)^2-\eta-C_2 u\}.
\end{aligned}
\end{equation}
Then
\begin{equation}\label{Oin}
B(\tau,v)\subset \bigcup_{s_1,s_2,s_3\in \Theta} C(s_3)\cap D(s_1,s_2).
\end{equation}

The number of elements is, for any $v$ with $\min\{v,u-v\}\leq u (1-2^{2/3})+\e^{2/3}$ of order $u^2/\tau$. Since $\tau\geq 1/2$ (for $\e\leq 2^{-2/3}-1/2$), the sum contains $\Or(u^2)$ many terms. Therefore, using the independence of $C(s_3)$ and $D(s_1,s_2)$,
\begin{equation}
\Pb(B(\tau,v))\leq C u^2 \max_{s_1,s_2,s_3\in \Theta} \Pb(C(s_3))\Pb(D(s_1,s_2)) .
\end{equation}
As $\gamma_1$ and $\gamma_2$ ``occur disjointly'', by the BK (Berg-Kesten) inequality (see e.g.\ Theorem~7 of~\cite{AGH18} for a statement applicable in the above scenario) we get
\begin{equation}\label{eqB40}
\begin{aligned}
\Pb(D(s_1,s_2))&\leq \Pb(L_{(0,0),{\cal S}_v}\geq 4\tau N-\tfrac1{\tau}(v-1)^22^{4/3}N^{1/3}+s_1 2^{4/3} N^{1/3})\\
&\times \Pb(L_{I(u),{\cal S}_v}\geq 4\tau N-\tfrac1{\tau} (u-v-1)^22^{4/3}N^{1/3}+s_2 2^{4/3} N^{1/3}).
\end{aligned}
\end{equation}
Set ${\cal D}_{\tau,v}=\{(\tau N+k,\tau N-k)| k\geq (v-1)(2N)^{2/3}\}$. Then $L_{(0,0),{\cal S}_v}\leq L_{(0,0),{\cal D}_{\tau,v}}$ and Lemma~\ref{lemUB2} (after rescaling $N\to \tau N$) leads to
\begin{equation}
 \Pb(L_{(0,0),{\cal S}_v}\geq 4\tau N-\tfrac1{\tau}(v-1)^22^{4/3}N^{1/3}+s_1 2^{4/3} N^{1/3})\leq C e^{-\frac43 \frac{s_1^{3/2}}{\tau^{1/2}}}
\end{equation}
and similarly for the bound on the second term in \eqref{eqB40}, so that we have
\begin{equation}
\Pb(D(s_1,s_2))\leq C e^{-\frac43 \frac{s_1^{3/2}+s_2^{3/2}}{\tau^{1/2}}},
\end{equation}
for $2<u\ll N^{1/9}$, $0<s_1,s_2\ll N^{2/9}$.

To estimate $\Pb(C(s_3))$ we divide the segment ${\cal S}_v$ into pieces of length $(1-\tau)^{2/3}(2N)^{2/3}$ to which we can apply a rescaled version of Proposition~\ref{PropUB5}. We have $2/(1-\tau)^{2/3}\leq 2/\e^{2/3}$ such pieces (we used $1-\tau\geq \e$). Using union bound we then get, for $s_3 (1-\tau)^{-1/3}\ll N^{2/3}$,
\begin{equation}\label{eq43}
\Pb(C(s_3))\leq \frac{C \max\{1,s_3 (1-\tau)^{-1/3}\}}{\e^{2/3}} e^{-\frac43 \frac{s_3^{3/2}}{(1-\tau)^{1/2}}}.
\end{equation}
Let $\e=\delta u^{-2}$ for $\delta>0$ small enough as in the proof of Lemma~\ref{lemBoundGamma1}. We have
\begin{equation}\label{Oub4}
\Pb(B(\tau,v))\leq C u^4 \delta^{-1}\max_{s_1,s_2,s_3\in\Theta} e^{-\frac43 \frac{s_1^{3/2}+s_2^{3/2}}{\tau^{1/2}}} e^{-\frac43 \frac{s_3^{3/2}}{(1-\tau)^{1/2}}}.
\end{equation}
Therefore we concentrate now on finding the maximum of
\begin{equation}
e^{-\frac43 \frac{s_1^{3/2}+s_2^{3/2}}{\tau^{1/2}}}e^{-\frac43 \frac{s_3^{3/2}}{(1-\tau)^{1/2}}}
\end{equation}
for $s_1,s_2,s_3\in \Theta$. Define $\tilde s_3=s_3+\eta+C_2 u$. Then for given $s_3$ on $\Theta$, we have
\begin{equation}
s_1=\frac{(v-1)^2}{\tau}-\tilde s_3,\quad s_2=\frac{(u-v-1)^2}{\tau}-\tilde s_3.
\end{equation}
So we need to maximize
\begin{equation}
M(v,\tau,s_3)=-\frac43 \frac{((v-1)^2/\tau-\tilde s_3)^{3/2} +((u-v-1)^2/\tau-\tilde s_3)^{3/2}}{\sqrt{\tau}}-\frac43 \frac{s_3^{3/2}}{\sqrt{1-\tau}}.
\end{equation}

In principle, to get the bound on $B(\tau,v)$, we would need to find $s_3$ maximizing $M(v,\tau,s_3)$. In the statement we want a bound uniform in $\tau,v$. This means that we need to maximize the result over $\tau,v$ as well. In short, we maximize $M$ for $s_3,v,\tau$ and thus we do it in another order. First notice that for given $\tau,s_3$, $M(v,\tau,s_3)$ is maximized at $v=u/2$, for which
\begin{equation}
M(u/2,\tau,s_3)=-\frac43 \frac{2((u-2)^2/(4\tau)-\tilde s_3)^{3/2}}{\sqrt{\tau}}-\frac43 \frac{s_3^{3/2}}{\sqrt{1-\tau}}.
\end{equation}
Computing the derivative with respect to $s_3$ we get that, for a given $\tau$, the maximum is at $s_3^*=\frac{[(u-2)^2-4\tau(\eta+C_2 u)](1-\tau)}{\tau(4-3\tau)}>0$ under the assumption $u\geq 4(1+C_2)$ and $\eta<1$. So we get
\begin{equation}\label{eq2.74}
\begin{aligned}
M(u/2,\tau,s_3^*)&=-\frac43 \frac{[(u-2)^2-4\tau (\eta+C_2 u)]^{3/2}}{4\tau^{3/2} \sqrt{4-3\tau}}\\
&\leq -\frac13 [(u-2)^2-4\tau (\eta+C_2 u)]^{3/2}\leq -\frac13 u^3+c u^2
\end{aligned}
\end{equation}
for some constant $c>0$. Inserting \eqref{eq2.74} into \eqref{Oub4} and choosing an appropriate new constant $c$ leads to the claimed result.
\end{proof}

\subsection{Proof of Theorem~\ref{thmUpperBound}}
As before, we shall prove the bound first for sufficiently large $u$, and adjust $c$ later to deduce the same for all $u>1$.

Recall that $L_N^*(u)$ is the rescaled LPP from $I(u)$ to $\LL_{2N}$, see \eqref{eq1.4}.
Let us use the notations $X=L_N^*(0)$ and $Y=L_N^*(u)$. For $j,j'\in \{1,2,\ldots, u-1\}$, let $S^{j}_N$ denote the weight of the maximum weight path $\pi^j$ from $(0,0)$ to $\LL_{2n}$ such that $\pi(t) < j(2N)^{2/3}$ for all $t$, and similarly, let $S^{u,j'}_N$ denote the weight of the maximum weight path $\tilde \pi^{j'}$ from $I(u)$ to $\LL_{2n}$ such that $\tilde \pi^{j'}(t) > (u-j')(2N)^{2/3}$ for all $t$. Let us also set
\begin{equation}
X_j=2^{-4/3}N^{-1/3}(S^{j}_N-4N), \quad Y_{j'}=2^{-4/3}N^{-1/3}(S_N^{u,j'}-4N).
\end{equation}
For notational convenience, we shall also write $X_0=Y_0=0$ and $X_{u}=X$, $Y_{u}=Y$.
Now, writing
\begin{equation}
X=\sum_{j=0}^{u-1} X_{j+1}-X_{j},\quad Y=\sum_{j'=0}^{u-1} Y_{j'+1}-Y_{j'},
\end{equation}
and using the bilinearity of covariance it is enough to prove that for some $c>0$ and for all $j, j'\in \{0,1,\ldots, u-1\}$
\begin{equation}\label{eq:jj}
{\rm Cov}(X_{j+1}-X_{j}, Y_{j'+1}-Y_{j'}) \le e^{cu^2}e^{-\frac{1}{3}u^3}.
\end{equation}

Notice first that $X_{j+1}-X_{j}$ and $Y_{j'+1}-Y_{j'}$ depend on disjoint sets of vertex weights and hence are independent unless $j+j'\ge u-1$. Hence we only need to consider $(j,j')$ such that $j+j'\ge u-1$. For such a pair, noticing $X\ge X_{j+1}\ge X_{j}$ and $Y\ge Y_{j+1}\ge Y_{j}$ it follows that
\begin{equation}
{\rm Cov}(X_{j+1}-X_{j}, Y_{j'+1}-Y_{j'}) \le \E[(X-X_j)(Y-Y_{j'})].
\end{equation}
For convenience of notation, let $\Gamma_1$ and $\Gamma_2$ locally denote the geodesics from $(0,0)$ and $I(u)$ respectively to $\LL_{2N}$. We define
\begin{equation}
A_{j}=\Big\{\sup_{0\leq t\leq 2N} \Gamma_1(t) \ge j(2N)^{2/3}\Big\},\quad B_{j'}=\Big\{\inf_{0\leq t\leq 2N} \Gamma_2(t) \le (u-j')(2N)^{2/3}\Big\}.
\end{equation}
Clearly, $(X-X_j)(Y-Y_{j'})=0$ on the complement of $A_{j}\cap B_{j'}$ and $X-X_{j}$ and $Y-Y_{j'}$ are positive random variable with super-exponential (uniform in $j,j'$) tails (indeed we can just use the upper tail bounds for $X$ and $Y$). Using the notation $\| X \|_p=\E(|X|^p)^{1/p}$ and the fact that the $p$-th norm of the of random variables with super-exponential tails can grow at most linearly in $p$, we know that there exists a constant $C$ such that for all $j,j'$ and all $p\ge 1$
$||X-X_{j}||_{p}, ||Y-Y_{j'}||_{p}\le Cp.$ Using the H\"older inequality we have
\begin{equation}
\E[(X-X_{j})(Y-Y_{j'})]\le ||\Id_{A_j\cap B_{j'}}||_{q}||(X-X_{j})(Y-Y_{j'})||_{p}
\end{equation}
where $p^{-1}+q^{-1}=1$.
By the Cauchy-Schwarz inequality
\begin{equation}
||(X-X_{j})(Y-Y_{j'})||_{p}\le ||X-X_{j}||_{2p}||Y-Y_{j'}||_{2p}\le C p^2
\end{equation} for some new constant $C>0$. It therefore follows that
\begin{equation}
\label{eq:momentbound}
\E[(X-X_{j})(Y-Y_{j'})]\le C'p^2 \Pb(A_{j}\cap B_{j'})^{1/q}.
\end{equation}
for $p,q \ge 1$ with $p^{-1}+q^{-1}=1$.

By Lemma~\ref{lemAjBj} below, we have
\begin{equation}
\label{eq:ajbjbound}
\Pb(A_{j}\cap B_{j'})\le e^{-\frac{1}{3}u^{3}+cu^2}
\end{equation}
for some $c>0$. We choose $p=u$ so that $1/q=1-\frac{1}{u}$. Therefore, plugging \eqref{eq:ajbjbound} in \eqref{eq:momentbound} it follows that
\begin{equation}
\E[(X-X_{j})(Y-Y_{j'})]\le Cu^2 e^{(cu^2-u^3/3)(1-\frac{1}{u})}\le e^{-\frac13 u^3+c' u^2}
\end{equation}
for some new constant $c'$. This establishes \eqref{eq:jj} and Theorem \ref{thmUpperBound} follows by summing over $(j,j')$.
\qed

\begin{lem}\label{lemAjBj}
In the above set-up, for $u$ large enough and $j+j'\ge u-1$ we have
\begin{equation}
\Pb(A_{j}\cap B_{j'})\le e^{-\frac{1}{3}u^{3}+cu^2}
\end{equation}
for some $c>0$.
\end{lem}

\begin{proof}
Notice first that arguing as in the proof of Lemma~\ref{PropNotCrossingFar}, if $j\ge 2^{-2/3}u$ then $\Pb(A_{j})\le e^{-\frac{1}{3}u^{3}+cu^2}$ and similarly if $j'\ge 2^{-2/3}u$ then $\Pb(B_{j'})\le e^{-\frac{1}{3}u^{3}+cu^2}$. Therefore it suffices to consider only the cases $\max\{j,j'\}\le 2^{-2/3}u$. This, together with $j+j'\ge u-1$ also implies that $\min\{j,j'\}\ge (1-2^{-2/3})u-1>0$ for $u$ sufficiently large.

Observe now that
\begin{equation}
\Pb(A_{j}\cap B_{j'})\le \Pb(\Gamma_{1}\cap \Gamma_2 \neq \emptyset)+\Pb(A_j\cap B_{j'}\cap \{\Gamma_{1}\cap \Gamma_2 =\emptyset\}).
\end{equation}
By Theorem~\ref{ThmCrossings} it follows that the first term is upper bounded by $e^{-\frac{1}{3}u^{3}+cu^2}$ and hence it suffices to show that
\begin{equation}
\label{eq:ajbj2}
\Pb(A_j\cap B_{j'}\cap \{\Gamma_{1}\cap \Gamma_2 =\emptyset\})\le e^{-\frac13 u^3+cu^2}
\end{equation}
for some $c>0$.

Since the two geodesics do not intersect, we would like to use the BK inequality to get an upper bound. However, we can not do it directly, since the property of being a geodesic depends on all the random variables and not on subsets. We therefore would show something similar for any paths which are of typical length. First, we need to approximate the events $A_j$ and $B_{j'}$.

For $\varepsilon>0$ to be chosen later, let $\tilde{A}_{j}$ denote the event that there exists $k\in \{1,2, \ldots ,\frac{1}{\varepsilon}\}$ such that $\Gamma_1(2k\varepsilon N)\ge (j-1)(2N)^{2/3}$. Similarly, let $\tilde{B}_{j'}$ denote the event that there exists $k\in \{1,2, \ldots ,\frac{1}{\varepsilon}\}$ such that $\Gamma_2(2k\varepsilon N)\le (u-j'+1)(2N)^{2/3}$. By choosing $\e=\delta u^{-3/2}$ for $\delta$ sufficiently small but fixed and arguing as in the proof of Theorem~\ref{ThmUB6Bis} it follows that
\begin{equation}
\Pb(A_{j} \setminus \tilde{A}_{j}), \Pb(B_{j'}\setminus \tilde B_{j'}) \le e^{-\frac{1}{3}u^{3}+c u^2}.
\end{equation}
Therefore \eqref{eq:ajbj2} reduces to showing
\begin{equation}
\label{eq:ajbj3}
\Pb(\tilde{A}_j\cap \tilde{B}_{j'}\cap \{\Gamma_{1}\cap \Gamma_2 =\emptyset\})\le e^{-\frac13 u^3+c u^2}
\end{equation}
for some $c>0$.

Let $C_2$ be such that (using Lemma~\ref{lemUB1})
\begin{equation}
\Pb(X\le -C_2 u)=\Pb(Y\le -C_2u) \le e^{-\frac{1}{3}u^{3}}.
\end{equation}
Observe that on $\tilde{A}_j\cap \tilde{B}_{j'}\cap \{\Gamma_{1}\cap \Gamma_2 =\emptyset\} \cap \{X>-C_2u\}\cap \{Y>-C_2u\}$ there exist disjoint paths $\gamma_1$ and $\gamma_2$ from $(0,0)$ and $I(u)$ respectively to $\LL_{2N}$ with $L(\gamma_1), L(\gamma_2)\ge 4N-C_2u 2^{4/3}N^{1/3}$ such that there exist $k_1,k_2\in \{1,2,\ldots, \frac{1}{\varepsilon}\}$ with $\gamma_1(2k_1\varepsilon N)\ge (j-1)(2N)^{2/3}$ and $\gamma_2(2k_2\varepsilon N)\leq (u-j'+1)(2N)^{2/3}$. By using the BK inequality as before we get that
\begin{equation}
\label{eq:ajbj4}
\Pb(\tilde{A}_j\cap \tilde{B}_{j'}\cap \{\Gamma_{1}\cap \Gamma_2 =\emptyset\}) \le \Pb(\hat{A}_{j})\Pb(\hat{B}_{j'})+2e^{-\frac{1}{3}u^{3}}
\end{equation}
where $\hat{A}_{j}$ denotes the event that there exists a path $\gamma_1$ from $(0,0)$ to $\LL_{2N}$ satisfying $L(\gamma_1)>4N-C_2u2^{4/3}N^{1/3}$ and $\gamma_1(2k\varepsilon N)\ge (j-1)(2N)^{2/3}$ for some $k$ and $\hat{B}_{j'}$ denotes the event that there exists a path $\gamma_2$ from $I(u)$ to $\LL_{2N}$ satisfying $L(\gamma_2)>4N-C_2u2^{4/3}N^{1/3}$ and $\gamma_2(2k\varepsilon N)\le (u-j-1)(2N)^{2/3}$ for some $k$. We claim that
\begin{equation}\label{eq2.92}
\Pb(\hat{A}_{j})\le e^{cu^2}e^{-\frac{4}{3}(j-1)^3},\quad \Pb(\hat{B}_{j}) \le e^{cu^2}e^{-\frac{4}{3}(j'-1)^3}.
\end{equation}
Postponing the proof of \eqref{eq2.92} for now, let us first complete the proof of the lemma. By Jensen's inequality together with $j+j'\ge u-1$ it follows that
\begin{equation}
(j-1)^3+ (j'-1)^3 \ge \frac{1}{4}(j+j'-2)^{3}\ge \frac{(u-3)^3}{4}
\end{equation}
and hence
\begin{equation}
\Pb(\hat{A}_{j})\Pb(\hat{B}_{j'})\le e^{-\frac{1}{3}u^3+c u^2}
\end{equation}
for some $c>0$. This, together with \eqref{eq:ajbj4} establishes \eqref{eq:ajbj3}.

To conclude the proof we show \eqref{eq2.92}. The idea is to follow the proof of Lemma~\ref{lemUB3B}. However the first bound \eqref{Oub5b} in that proof applies only to geodesics, while here we have to show it for any paths with a length larger that $4N-C_2 u 2^{4/3}N^{1/3}$. We will prove that for any path $\gamma_1$ satisfying the conditions of $\hat A_j$, for any $\tau\in\{\e,2\e,\ldots,1\}$
\begin{equation}\label{eq2.90}
\Pb(\gamma_1(2 \tau N)\geq M (2N)^{1/3})\leq C u^2 e^{-\frac43 u^3}
\end{equation}
for $M=\sqrt{C_2 u+ u^2}$. Then the rest of the proof of Lemma~\ref{lemUB3B} applies, except that the sum in \eqref{eqB7} goes until $M-u-1$.

Denote $K(v)=(\tau N,\tau N)+v(2N)^{2/3}(1,-1)$ and divide the possible points where $\gamma_1$ crosses the line $\LL_{2\tau N}$ as
\begin{equation}
\{K(v),v\in [M,\tfrac{\tau N}{(2N)^{2/3}}]\}={\cal I}_0\bigcup_{\ell=M}^{N^{1/10}-1} {\cal I}_\ell,
\end{equation}
with
${\cal I}_0=\{K(v),v\in [N^{1/10},\tfrac{\tau N}{(2N)^{2/3}}]\}$ and ${\cal I}_\ell=\{K(v), v\in [\ell,\ell+1)\}$.
Then we have, for any choice of $A_\ell$ and $B_\ell=4N-C_2 u 2^{4/3}N^{1/3}-A_\ell$,
\begin{equation}
\begin{aligned}
\Pb(\gamma_1(2 \tau N)\geq M (2N)^{1/3})&\leq \sum_{\ell} \Pb\left(L_{(0,0),{\cal I}_\ell}+L_{{\cal I}_\ell,\LL_{2N}}\geq 4N-C_2 u 2^{4/3}N^{1/3}\right)\\
&\leq \sum_\ell \Pb\left(L_{(0,0),{\cal I}_\ell}\geq A_\ell\right)+\Pb\left(L_{(0,0),{\cal I}_\ell}\geq B_\ell\right).
\end{aligned}
\end{equation}
By rescaling Lemma~\ref{lemUB2}, for $s_1\ll (\tau N)^{2/9}$ and $u\ll (\tau N)^{1/9}$
\begin{equation}
\begin{aligned}
\Pb\Big(L_{(0,0),{\cal I}_\ell}>4\tau N-\frac{\ell^2}{\tau} 2^{4/3} N^{1/3}+ s_1 \tau^{1/3}2^{4/3}N^{1/3}\Big)&\leq C e^{-\frac43 s_1^{3/2}},\\
\Pb\Big(L_{(0,0),{\cal I}_0}>4\tau N-\frac{N^{1/5}}{\tau} 2^{4/3} N^{1/3}+s_1 \tau^{1/3}2^{4/3}N^{1/3}\Big)&\leq C e^{-\frac43 s_1^{3/2}},
\end{aligned}
\end{equation}
and by rescaling Proposition~\ref{PropUB5} (and union bound on $(1-\tau)^{-2/3}$ subsegments per $(2N)^{2/3}$-length), for $s_2\ll (1-\tau)^{2/3}N^{2/3}$, we get
\begin{equation}
\begin{aligned}
\Pb\Big(L_{{\cal I}_\ell,\LL_{2N}}>4(1-\tau) N+s_2 (1-\tau)^{1/3}2^{4/3}N^{1/3}\Big)&\leq C s_2(1-\tau)^{-2/3} e^{-\frac43 s_2^{3/2}},\\
\Pb\Big(L_{{\cal I}_0,\LL_{2N}}>4(1-\tau) N+ s_2 (1-\tau)^{1/3}2^{4/3}N^{1/3}\Big)&\leq C s_2 N^{1/3}(1-\tau)^{-2/3} e^{-\frac43 s_2^{3/2}}.
\end{aligned}
\end{equation}

We take, with $\alpha_\ell=-\frac{\ell^2 (1-\tau)^{1/3}+C_2 u \tau^{4/3}}{((1-\tau)^{1/3}+\tau^{1/3})\tau}$,
$A_\ell=4\tau N+\alpha_\ell 2^{4/3}n^{1/3}$. Setting $\ell=M+\tilde \ell$, we get
\begin{equation}
s_1=s_2=\frac{\ell^2-C_2 u \tau}{((1-\tau)^{1/3}+\tau^{1/3})\tau}\geq u^2+\tilde \ell^2.
\end{equation}
From this, it follows that
\begin{equation}\label{eq2.100}
\Pb\left(L_{(0,0),{\cal I}_\ell}\geq A_\ell\right)+\Pb\left(L_{(0,0),{\cal I}_\ell}\geq B_\ell\right)\leq 2C \delta^{-2/3} u \ell^2 e^{-\frac43 u^3}e^{-\frac43 (\ell-M)^{3/2}}
\end{equation}
and thus $\sum_{\ell=M}^{N^{1/10}-1}\eqref{eq2.100}\leq C' \delta^{-2/3}u^2 e^{-\frac43 u^3}$, while for $\ell=0$ it is of an order $\Or(e^{-\frac43 N^{3/20}})$ smaller. Applying union bound on the $\e^{-1}=u^{3/2}/\delta$ time intervals and the estimate \eqref{eq2.90}, we get that any path in $\hat A_j$ is localized within a distance $M(2N)^{1/3}$ with probability at least $1-C e^{-u^3}$.
\end{proof}

\section{Lower Bound}\label{SectLowerBound}
In this section we prove the lower bound of Theorem~\ref{ThmMain}.
	\begin{thm}\label{ThmLowerBound}
There exist a constant $c>0$ such that
		\begin{equation}
			\Cov({\cal A}_1(0),{\cal A}_1(u))\geq e^{-c u \ln(u)} e^{-\frac43 u^3}.
		\end{equation}
	\end{thm}
We begin explaining the strategy of the proof. Hoeffding's covariance identity~\cite{Hoef40}, which comes from integration by parts on $\R_+$ and $\R_-$ separately, states that
\begin{equation}\label{covid}
	\Cov(X,Y)=\int_\R ds_1\int_\R ds_2\left[\Pb(X\leq s_1,Y\leq s_2)-\Pb(X\leq s_1)\Pb(Y\leq s_2)\right].
\end{equation}
We therefore define the following functions
\begin{equation}\label{Fs}
	\begin{aligned}
	F(u;s_1,s_2)&=\Pb\left({\cal A}_1(0)\leq s_1,{\cal A}_1(u)\leq s_2\right),\\
	f(s)&=\Pb\left({\cal A}_1(0)\leq s\right)=\Pb\left({\cal A}_1(u)\leq s\right),
\end{aligned}
\end{equation}
where in the last equality we used the stationarity of ${\cal A}_1$. As we would like to use \eqref{covid} with $X={\cal A}_1(0)$ and $Y={\cal A}_1(u)$, we are interested in finding the asymptotic behavior of
\begin{equation}\label{Ff2}
	F(u;s_1,s_2)-f(s_1)f(s_2) \quad \text{ as } u\rightarrow\infty.
\end{equation}
As $u\to\infty$, the random variables ${\cal A}_1(0)$ and ${\cal A}_1(u)$ become independent of each other; thus we define $\cal E$ through
\begin{equation}\label{Ff}
	F(u;s_1,s_2)=f(s_1)f(s_2)(1+{\cal E}(u;s_1,s_2)),
\end{equation}
where ${\cal E}\rightarrow 0$ when $u\rightarrow \infty$ (at least for $s_1$ and $s_2$ independent of $u$). Using \eqref{Ff} in \eqref{Ff2} and \eqref{covid} we obtain
\begin{equation}\label{int}
	\text{Cov}({\cal A}_1(0),{\cal A}_1(u))=\int_\R ds_1 \int_\R ds_2 f(s_1)f(s_2){\cal E}(u;s_1,s_2)
\end{equation}
Next, by the FKG inequality, see Lemma~\ref{lemFKG} below, the integrand in \eqref{int} is positive for all $u\geq 0$. We can therefore restrict the integration in \eqref{int} to a compact subset of $\R^2$ to obtain the following lower bound
\begin{equation}\label{eq1.7}
	\text{Cov}({\cal A}_1(0),{\cal A}_1(u))\geq \int_\alpha^{\beta} ds_1 \int_\alpha^\beta ds_2 f(s_1)f(s_2) {\cal E}(u;s_1,s_2)
\end{equation}
for any choice of $\alpha<\beta$.

Thus the goal of the computations below is to show that $\cal E$ is of order $e^{-\frac43 u^3}$ times a subleading term for $s_1,s_2$ in some chosen intervals of size $\Or(1)$ where $f(s_1)$ and $f(s_2)$ are bounded away from $0$.

\begin{lem}\label{lemFKG}
For any $s_1,s_2\in\R$,
	\begin{equation}\label{fkga}
		\Pb({\cal A}_1(0)\leq s_1,{\cal A}_1(u)\leq s_2)-\Pb({\cal A}_1(0)\leq s_1)\Pb({\cal A}_1(u)\leq s_2)\geq 0.
	\end{equation}
\end{lem}
\begin{proof}
Recalling the notation from \eqref{eq1.4}, notice that both $L^{*}_{N}(0)$ and $L_{N}^{*}(u)$ are increasing function of the weights $\omega_{i,j}\sim \exp(1)$ (and they depend only on finitely many vertex weights). For any $t_1,t_2$ it therefore follows that the events $\{L^{*}_{N}(0)\le t_1\}$ and $\{L^{*}_{N}(u)\le t_2\}$ are both decreasing and hence by the FKG inequality they are positively correlated (note that the FKG inequality is often stated for measures on finite distributive lattices satisfying the FKG lattice condition, but more general versions for product measures on finite products of totally ordered measure spaces applicable in the above scenario are available; see e.g.\ Lemma~2.1 of~\cite{Kesten2003} or Corollary~2 of~\cite{kemperman1977fkg}), and therefore
\begin{equation}\label{fkg}
\Pb(L^{*}_{N}(0)\le t_1,L^{*}_{N}(u)\le t_2)-\Pb(L^{*}_{N}(0)\le t_1)\Pb(L^{*}_{N}(u)\le t_2)\geq 0.
\end{equation}
Using
that the Airy$_1$ process is a scaling limit of $L^*$ (see~\eqref{eqUB5}), the proof is complete.
\end{proof}

We first derive an expression for ${\cal E}(u;s_1,s_2)$. Let us begin with the Fredholm representation of the function $F$. We have from~\cite{Sas05,BFPS06,Fer07}
	\begin{equation}\label{Fr}
	F(u;s_1,s_2)=\det(\Id-K)
\end{equation}
where $K$ is a $2\times 2$ matrix kernel
\begin{equation}\label{K}
	K=\left(
	\begin{array}{cc}
		K_{1,1} & K_{1,2} \\
		K_{2,1} & K_{2,2} \\
	\end{array}
	\right)
\end{equation}
 with entries given by the extended kernel of the Airy$_1$ process~\cite{Sas05,BFPS06,Fer07}
 \begin{equation}\label{Ker}
	\begin{aligned}
		K_{1,1}(x,y)&=\Id_{[x>s_1]}\Id_{[y>s_1]}\Ai(x+y),\\
		K_{1,2}(x,y)&=\Id_{[x>s_1]}\Id_{[y>s_2]}\Big(\Ai(x+y+u^2) e^{(x+y)u+\tfrac23 u^3}-\frac{e^{-(x-y)^2/4u}}{\sqrt{4\pi u}}\Big),\\
		K_{2,1}(x,y)&=\Id_{[x>s_2]}\Id_{[y>s_1]}\Ai(x+y+u^2) e^{-(x+y)u-\tfrac23 u^3},\\
		K_{2,2}(x,y)&=\Id_{[x>s_2]}\Id_{[y>s_2]}\Ai(x+y),
	\end{aligned}
\end{equation}
where $\Ai$ denotes the Airy function. Also, for the one-point distributions, we have $f(s_i)=\det(\Id-K_{i,i})$ for $i=1,2$.

The first step is the following result.
\begin{lem}\label{LemmaDec}
With the above notations
\begin{equation}\label{FdK}
1+{\cal E}(u;s_1,s_2)=\det(\Id-\widetilde K)
\end{equation}
where
\begin{equation}
	\widetilde K=(\Id-K_{1,1})^{-1}K_{1,2}(\Id-K_{2,2})^{-1}K_{2,1}.
\end{equation}
\end{lem}
\begin{proof}
Similar to~\cite{Wid03}, we compute
\begin{equation}
		\begin{aligned}
			& \det\left(\Id-\left(
			\begin{array}{cc}
				K_{1,1} & K_{1,2} \\
				K_{2,1} & K_{2,2} \\
			\end{array}
			\right)
			\right)
			= \det\left(\Id-\left(
			\begin{array}{cc}
				K_{1,1} & 0 \\
				0 & K_{2,2} \\
			\end{array}
			\right)
			-
			\left(
			\begin{array}{cc}
				0 & K_{1,2} \\
				K_{2,1} & 0 \\
			\end{array}
			\right)
			\right)\\
			&= \det\left(\Id-\left(
			\begin{array}{cc}
				K_{1,1} & 0 \\
				0 & K_{2,2} \\
			\end{array}
			\right)\right) \\&\qquad \times
			\det\left(\Id-\left(
			\begin{array}{cc}
				(\Id-K_{1,1})^{-1} & 0 \\
				0 & (\Id-K_{2,2})^{-1} \\
			\end{array}
			\right)
			\left(
			\begin{array}{cc}
				0 & K_{1,2} \\
				K_{2,1} & 0 \\
			\end{array}
			\right)\right) \\
			&=\det(\Id-K_{1,1})\det(\Id-K_{2,2})
			\det\left(\Id-\left(
			\begin{array}{cc}
				0 &-G \\
				-H & 0 \\
			\end{array}
			\right)\right),
		\end{aligned}
	\end{equation}
	where we set $G=-(\Id-K_{1,1})^{-1}K_{1,2}$ and $H=-(\Id-K_{2,2})^{-1}K_{2,1}$. Moreover,
	\begin{equation}
		\begin{aligned}
			\det\left( \begin{array}{cc}
				\Id &G \\
				H & \Id \\
			\end{array}
			\right)
			=
			\det\left(\left(
			\begin{array}{cc}
				\Id & G \\
				H & \Id \\
			\end{array}
			\right)
			\left(
			\begin{array}{cc}
				\Id & 0 \\
				-H & \Id \\
			\end{array}
			\right)
			\right)=\det(\Id-G\, H),
		\end{aligned}
	\end{equation}
	where
	\begin{equation}
		G\, H = (\Id-K_{1,1})^{-1}K_{1,2}(\Id-K_{2,2})^{-1}K_{2,1}.
	\end{equation}
	Since $\det(\Id-K_{\ell,\ell})=f(s_\ell)$, for $\ell=1,2$ as we mentioned above, \eqref{FdK} follows from \eqref{Ff}.
\end{proof}

Next we would like to approximate the Fredholm determinant in \eqref{FdK} by that of a simpler kernel.
	\begin{prop}\label{PropR1}
Let us define
\begin{equation}
R_1(u;s_1,s_2):=\det(\Id-\tilde{K})-\det(\Id-K_{1,2}K_{2,1}).
\end{equation}
Then, for any $s_1,s_2\geq 0$, there exists a constant $C_2>0$ such that
		\begin{equation}\label{ubd}
			|R_1(u;s_1,s_2)|\leq \frac{C_2 e^{-\min\{s_1,s_2\}}}{u^2}e^{-\tfrac43 u^3-2(s_1+s_2)u}.
		\end{equation}
for any $u\geq \max\{\frac12,\sqrt{s_1+s_2}\}$.
	\end{prop}
The proof of this proposition is in Section~\ref{sectProofProp34}.

\subsection{The leading term}\label{sec:tr}
Lemma~\ref{LemmaDec} and Proposition~\ref{PropR1} suggest
\begin{equation}
{\cal E}(u;s_1,s_1)\sim \det(\Id-K_{1,2}K_{2,1})-1 \quad \text{ as }u\rightarrow \infty.
\end{equation}
For a trace class operator $\cal K$, the Fredholm determinant is given by
\begin{equation}
	\det(\Id-{\cal K})=\sum_{n=0}^{\infty}\frac{(-1)^n}{n!}\int_{\R^n} dx_1\cdots dx_n\det[{\cal K}(x_i,x_j)]_{i,j=1}^n.
\end{equation}
When ${\cal K}=K_{1,2}K_{2,1}$, this translates to
\begin{equation}\label{asy2}
\det(\Id-K_{1,2}K_{2,1})= 1-\Tr\big(K_{1,2}K_{2,1}\big)+R_2(u;s_1,s_2),
\end{equation}
where $\Tr\big(K_{1,2}K_{2,1}\big)=\int_\R dx \int_\R dy K_{1,2}(x,y)K_{2,1}(y,x)$ and
\begin{equation} \label{Req}
R_2(u;s_1,s_2)=\sum_{n=2}^{\infty}\frac{(-1)^n}{n!}\int_{s_1}^\infty\cdots\int_{s_1}^\infty dx_1\cdots dx_n \det[K_{1,2}K_{2,1}(x_i,x_j)]_{i,j=1}^n.
\end{equation}
From \eqref{Ker}, it is clear that the upper tail of $K_{1,2}K_{2,1}$ in either variables, is determined by that of the function $\Ai$, which is known to decay super-exponentially, see \eqref{asy}. As each of the determinants in \eqref{Req} is a sum of products of elements of the order of $\Tr\big(K_{1,2}K_{2,1}\big)$, one expects the latter to dominate $R_2$ and therefore that
\begin{equation}\label{tra}
	\det(\Id-K_{1,2}K_{2,1})\sim 1-\Tr(K_{1,2}K_{2,1})
\end{equation}
if $\Tr(K_{1,2}K_{2,1})$ is small.

Let us move on to the computation of $\Tr(K_{1,2}K_{2,1})$. We write the kernel entries $(K_{1,2} K_{2,1})(x,y)$ as well as its trace using complex integral representations, which will then be analyzed.
We start with the following identities (see e.g.\ Appendix~A of~\cite{BFP09} for the first and last, while the second is a standard Gaussian integral)
\begin{equation}\label{eqintRepr}
\begin{aligned}
		&\frac{1}{2\pi\I}\int_{\gamma_a} d\xi e^{-\xi^3/3+u\xi^2+x\xi}=\Ai(x+u^2)e^{\frac23u^3+ux},\\
		& \frac1{2\pi\I}\int_{\gamma_a} d\xi e^{u\xi^2+x\xi}= \frac{e^{-\frac{x^2}{4u}}}{\sqrt{4\pi u}},\\
		&\frac{1}{2\pi\I} \int_{\gamma_b} d\eta e^{\frac{\eta^3}3-u\eta^2-x\eta} = \Ai(x+u^2)e^{-\frac23 u^3-ux},
\end{aligned}
\end{equation}
where
\begin{align}
	\gamma_a&=a+\I\R \quad \text{for $a<u$}\label{contour1}\\
	\gamma_b&=b+\I\R \quad \text{for $b>u$}\label{contour2}.
\end{align}
Plugging these identities into \eqref{Ker} we can write
\begin{equation}\label{eq1.24}
\begin{aligned}
		K_{1,2}(x,y)&=\Id_{x>s_1}\Id_{y>s_2}\frac{1}{2\pi\I}\int_{\gamma_a} d\xi \Big[e^{-\xi^3/3+u\xi^2+(x+y)\xi}-e^{u\xi^2+(x-y)\xi}\Big],\\
		K_{2,1}(x,y)&=\Id_{x>s_2}\Id_{y>s_1}\frac{1}{2\pi\I} \int_{\gamma_b} d\eta e^{\frac{\eta^3}3-u\eta^2-(x+y)\eta}.
\end{aligned}
\end{equation}
This leads to the following representation of the kernel $K_{1,2}K_{2,1}$:
\begin{equation}
	\begin{aligned}\label{eq7}
		&(K_{1,2}K_{2,1})(x,y)=\int_\R dz K_{1,2}(x,z) K_{2,1}(z,y)\\
		&=\frac{\Id_{x>s_1}\Id_{y>s_1}}{(2\pi\I)^2}\int_{\gamma_a} d\xi\int_{\gamma_b} d\eta e^{u\xi^2+x\xi}e^{\frac{\eta^3}3-u\eta^2-y\eta}\Bigg[\frac{e^{-\frac{\xi^3}3-s_2(\eta-\xi)}}{\eta-\xi}-\frac{e^{-s_2(\eta+\xi)}}{\eta+\xi}\Bigg].
	\end{aligned}
\end{equation}
Setting $x=y$ and integrating over $x$ (on $[s_1,\infty)$ due to the indicator functions) we get
\begin{equation}
	\begin{aligned}\label{K12K21}
		\Tr(K_{1,2}K_{2,1})=& \frac{1}{(2\pi\I)^2}\int_{\gamma_a} d\xi \int_{\gamma_b} d\eta\, e^{\frac{\eta^3}3-u\eta^2-(s_1+s_2)\eta}\frac1{\eta-\xi}\\
		&\times\Big[e^{-\xi^3/3+u\xi^2 +(s_1+s_2)\xi}\frac{1}{\eta-\xi}-e^{u\xi^2+(s_1-s_2)\xi}\frac{1}{\xi+\eta}\Big].
	\end{aligned}
\end{equation}

We begin with determining the asymptotic behavior of $\Tr(K_{1,2}K_{2,1})$.
\begin{prop}\label{propTrace} For all $0\leq s_1,s_2\leq \sqrt{u}$,
	\begin{equation}\label{eq3.23}
		-\Tr(K_{1,2}K_{2,1})= \frac{1}{16\pi u^4}e^{-2(s_1+s_2)u-\frac43 u^3}\left[s_1s_2+\Or(u^{-1/4})\right]
	\end{equation}
as $u\to\infty$.
\end{prop}
\begin{proof}
To get the asymptotic behavior of the trace, we need to choose the parameters $a,b$ in the integration contours. We do it in a way that the dominant parts in the exponents of the contour integrals in \eqref{K12K21} are minimized.
\begin{equation}\label{eq3.32}
\begin{aligned}
    \eqref{K12K21}&= \frac{1}{(2\pi\I)^2}\int_{\gamma_a} d\xi_1 \int_{\gamma_b} d\eta\, e^{\frac{\eta^3}3-u\eta^2-(s_1+s_2)\eta}e^{-\xi_1^3/3+u\xi_1^2 +(s_1+s_2)\xi_1}\frac1{(\eta-\xi_1)^2}\\
	&-\frac{1}{(2\pi\I)^2}\int_{\gamma_a} d\xi_2 \int_{\gamma_b} d\eta\, e^{\frac{\eta^3}3-u\eta^2-(s_1+s_2)\eta}e^{u\xi_2^2+(s_1-s_2)\xi_2}\frac{1}{(\xi_2+\eta)(\eta-\xi_2)}.
\end{aligned}
\end{equation}

Now we need to choose the parameters $a,b$. For that reason we search for the minimizers of the different exponents in \eqref{eq3.32}:
	\begin{enumerate}
		\item[(a)] for
		\begin{equation}
			\frac{d}{d\xi_1}\Big(-\frac{\xi_1^3}3+u\xi_1^2+(s_1+s_2)\xi_1\Big)=2u\xi_1-\xi_1^2=0,
		\end{equation}
		which is solved for $\xi_1=u-\sqrt{u^2+s_1+s_2}=:a_1$ or $\xi_1=u+\sqrt{u^2+s_1+s_2}$. The solution which satisfy the constraint $\Re(\xi_1)<u$ in \eqref{contour1} is also the minimum.
		\item[(b)] For
		\begin{equation}
			\frac{d}{d\xi_2}\Big(u\xi_2^2+(s_1-s_2)\xi_2\Big)=2u\xi_2+(s_1-s_2)
		\end{equation}
we see that $\xi_2=(s_2-s_1)/(2u)=:a_2$ is the minimum.
		\item[(c)] Similarly,
		\begin{equation}
			\frac{d}{d\eta}\Big(\frac{\eta^3}3-u\eta^2-(s_1+s_2)\eta\Big)=\eta^2-2u\eta-(s_1+s_2)=0,
		\end{equation}
		has the minimum is at $\eta=u+\sqrt{u^2+s_1+s_2}=:b$ satisfies ${\rm{Re}(\eta)}>u$.
	\end{enumerate}

So let us use the following change of variables
\begin{equation}\label{chov}
\xi_1=a_1+\frac{z}{\sqrt{u}},\quad \xi_2=a_2+\frac{z}{\sqrt{u}},\quad \eta=b+\frac{w}{\sqrt{u}}
\end{equation}
with $z,w\in\I\R$ into \eqref{eq3.32}. The two terms are analyzed in the same way, thus we write the details only for the first one.

Denote $\sigma=(s_1+s_2)/u^2\leq 2 u^{-3/2}$ and consider the first term in \eqref{eq3.32}. We have
\begin{equation}
\begin{aligned}
&e^{\frac{\eta^3}3-u\eta^2-(s_1+s_2)\eta}e^{-\frac13\xi_1^3+u\xi_1^2 +(s_1+s_2)\xi_1}= e^{-\frac43u^3 (1+\sigma)^{3/2}} e^{\sqrt{1+\sigma}(w^2+z^2)} e^{\frac{w^3-z^3}{3 u^{3/2}}}\\
&=e^{-\frac43 u^3-2(s_1+s_2)u-\frac{(s_1+s_2)^2}{2u}(1+\Or(\sigma))} e^{(z^2+w^2)(1+\Or(\sigma;z u^{-3/2};w u^{-3/2}))}
\end{aligned}
\end{equation}
and for the prefactor\footnote{The notation $\Or(a_1;...\,;a_k)$ stands for $\Or(a_1)+...+\Or(a_k)$.}
\begin{equation}
\frac1{(\eta-\xi_1)^2} = \frac{1}{4u^2}(1+\Or(z u^{-3/2};w u^{-3/2};\sigma)).
\end{equation}
Each term in the exponential which is cubic in $z,w\in\I\R$ is purely imaginary, thus its exponential is bounded by $1$. Furthermore, the quadratic terms in $z$ and $w$ have a positive prefactor $\sqrt{1+\sigma}\geq 1$. Thus integrating over $|z|>u^{1/4}$ and/or $|w|>u^{1/4}$ we get a correction term of order $e^{-\sqrt{u}}$ times the value of the integrand at $z=w=0$. For the rest of the integral, for which $|z|,|w|\leq u^{1/4}$, the error terms $\Or(z u^{-3/2};w u^{-3/2};\sigma)=\Or(u^{-5/4})$.

So the first term in \eqref{eq3.32} is given by
\begin{multline}\label{eq3.39}
\frac{e^{-\frac43 u^3-2(s_1+s_2)u-\frac{(s_1+s_2)^2}{2u}(1+\Or(\sigma))}}{4 u^3} \bigg[\Or(e^{-\sqrt{u}})\\ +\frac{(1+\Or(u^{-5/4}))}{(2\pi\I)^2}\int_{-\I u^{1/4}}^{\I  u^{1/4}} dz \int_{-\I  u^{1/4}}^{\I  u^{1/4}} dw
e^{(z^2+w^2)(1+\Or(u^{-5/4}))}\bigg].
\end{multline}
Finally, extending the Gaussian integrals to $\I\R$, we get an error term $\Or(e^{-\sqrt{u}})$ only and using
\begin{equation}
\frac1{(2\pi\I)^2}\int_{\I\R}dz \int_{\I\R}dw\,e^{w^2+z^2}=\frac1{4\pi}
\end{equation}
we get
\begin{equation}\label{eq3.41}
\begin{aligned}
\eqref{eq3.39}&=\frac{e^{-\frac43 u^3-2(s_1+s_2)u-\frac{(s_1+s_2)^2}{2u}(1+\Or(\sigma))}}{16\pi u^3}(1+\Or(u^{-5/4}))\\
&=\frac{e^{-\frac43 u^3-2(s_1+s_2)u}}{16 \pi u^3}\left[1-\frac{(s_1+s_2)^2}{2u}+\Or(u^{-5/4})\right]
\end{aligned}
\end{equation}
A similar computation for the second term in \eqref{eq3.32} leads to
\begin{equation}\label{eq3.42}
-\frac{e^{-\frac43 u^3-2(s_1+s_2)u}}{16 \pi u^3}\left[1-\frac{s_1^2+s_2^2}{2u}+\Or(u^{-5/4})\right].
\end{equation}
Summing up \eqref{eq3.41} and \eqref{eq3.42} we obtain
\begin{equation}
-\Tr(K_{1,2}K_{2,1})=\frac{e^{-\frac43 u^3-2(s_1+s_2)u}}{16 \pi u^4}\left[s_1s_2+ \Or(u^{-1/4})\right].
\end{equation}
\end{proof}

\subsection{Bounding lower order terms}
To show \eqref{tra} we need to get a bound on $R_2(u;s_1,s_2)$ from \eqref{Req}, which is $o(\Tr\big(K_{1,2}K_{2,1}\big))$.

\begin{prop}\label{propR2}
There exists a constant $C>0$ such that
\begin{equation}
|R_2(u;s_1,s_2)|\leq \frac{C}{u^6} e^{-4(s_1+s_2)u} e^{-\frac83 u^3}
\end{equation}
uniformly in $s_1,s_2\in \R$.
\end{prop}
\begin{proof}
In \eqref{eq7} we use the change of variables
	\begin{equation}
		\xi = \frac{z}{\sqrt{u}},\quad
		\eta =2u+\frac{w}{\sqrt{u}}\label{chov4},
	\end{equation}
where $z,w\in\I\R$. This leads to
\begin{equation}\label{eq1.40}
\eqref{eq7} = \Id_{x>s_1}\Id_{y>s_1}\frac{e^{-\frac43 u^3-2u (s_2+y)}}{(2\pi\I)^2 u}\int_{\I\R}dz\int_{\I\R}dw e^{z^2+w^2} P(z,w,x,y)
\end{equation}
with
\begin{equation}
P(z,w,x,y)=e^{\frac{x z}{\sqrt{u}}+\frac{w^3}{3 u^{3/2}}-\frac{w y}{\sqrt{u}}} \bigg[
\frac{e^{-\frac{z^3}{3 u^{3/2}}-s_2\frac{w-z}{\sqrt{u}}}}{2u+\frac{w-z}{\sqrt{u}}}-\frac{e^{-s_2\frac{z+w}{\sqrt{u}}}}{2u+\frac{z+w}{\sqrt{u}}}
\bigg].
\end{equation}
Since $x,y\in\R$ and $z,w\in\I\R$, we get the simple bound $|P(z,w,x,y)|\leq \frac{1}{u}$, while $e^{z^2+w^2}$ is real. Performing the Gaussian integral we then get
\begin{equation}
|\eqref{eq1.40}|\leq \Id_{x>s_1}\Id_{y>s_1}\frac{e^{-\frac43 u^3-2 (s_2+y)u}}{4\pi u^2}.
\end{equation}

Let $K=K_{1,2}K_{2,1}$. Hadamard's inequality\footnote{Let $A$ be a $n\times n$ matrix with $|A_{i,j}|\leq 1$. Then $|\det(A)|\leq n^{n/2}$.} gives
	\begin{equation}
		|\det[K(x_i,x_j)]_{i,j=1}^n|\leq n^{n/2} \prod_{j=1}^n \frac{e^{-\frac{4u^3}{3}-2(x_j+s_2)u}\Id_{x_j>s_1}}{4\pi u^2}
	\end{equation}
	so that
	\begin{equation}
		\bigg|\int_{x_1\geq s_1}dx_1\cdots\int_{x_n\geq s_1}dx_n\det[K(x_i,x_j)]_{i,j=1}^n \bigg|\leq n^{n/2}\bigg(\frac{e^{-\frac{{4u^3}}{3}-2(s_1+s_2)u}}{8\pi u^3}\bigg)^n.
	\end{equation}	
Denote $M=\frac{e^{-\frac{{4u^3}}{3}-2(s_1+s_2)u}}{8\pi u^3}$. Then there exists $C>0$ such that
	\begin{equation}
		|R_2(u;s_1,s_2)|\leq\sum_{n=2}^{\infty}\frac{M^nn^{n/2}}{n!}\leq C M^2\leq Cu^{-6}e^{-\frac{{8u^3}}{3}-4(s_1+s_2)u}.
	\end{equation}
	This completes the proof.
\end{proof}

We have now all the ingredients to complete the proof of Theorem~\ref{ThmLowerBound}.
\begin{proof}[Proof of Theorem~\ref{ThmLowerBound}]
We have
	\begin{equation}
		\begin{aligned}
			F(u;s_1,s_2)\stackrel{\eqref{FdK}}{=}&f(s_1)f(s_2)\det(\Id-\tilde{K})\\
			\stackrel{\eqref{ubd}}{=}&f(s_1)f(s_2)[\det(\Id-K_{1,2}K_{2,1})+R_1(u;s_1,s_2)]\\
			\stackrel{\eqref{asy2}}{=}&f(s_1)f(s_2)[1-\Tr(K_{1,2}K_{2,1})+R(u;s_1,s_2)],
		\end{aligned}
	\end{equation}
where
\begin{equation}
R(u;s_1,s_2)=R_1(u;s_1,s_2)+R_2(u;s_1,s_2).
\end{equation}
It follows that
\begin{equation}
	F(u;s_1,s_2)-f(s_1)f(s_2)=-f(s_1)f(s_2)\Tr(K_{1,2}K_{2,1})+f(s_1)f(s_2)R(u;s_1,s_2).
\end{equation}
We shall apply the lower bound \eqref{eq1.7} for the covariance for an appropriate choice of $\alpha=\alpha(u)$ and $\beta=\beta(u)$ such the contribution of error term $f(s_1)f(s_2)R(u;s_1,s_2)$ is of a subleading order.

We shall choose $\alpha>0$ depending on $u$ and for concreteness set $\beta=\alpha+1$. Thus we integrate over a region of area $1$. By Proposition~\ref{propR2}, the error term $R_2(u;s_1,s_2)$ is much smaller than the leading term coming from the trace, see Proposition~\ref{propTrace} (of order $e^{-\frac43 u^3}$ smaller) for $s_1,s_2\ll u^2$. However, the exponential term in $u$ in the error bound of $R_1(u;s_1,s_2)$ coming from \eqref{ubd} is of the same order, namely $e^{-2(s_1+s_2)u-\frac43 u^3}$. The difference is that in the trace we have a prefactor $\sim \frac{s_1 s_2}{u^4}$, while in the bound for $R_1$ we have a prefactor $\sim\frac{e^{-\min\{s_1,s_2\}}}{u^2}$. Thus, in order to ensure that the contribution from $R_1$ is subleading, we can take $s_1,s_2\sim c\ln(u)$ for any $c\geq 3$. Therefore choosing $\alpha=3\ln(u)$ and using the fact that for all $x\geq 0$,  $(x)=F_{\rm GOE}(2^{2/3}x)$ is bounded away from 0 (in fact, one can numerically check $f(x)=F_{\rm GOE}(2^{2/3}x)\in [F_{\rm GOE}(0),1]=[0.83...,1]$), we finally get the claimed result.
\end{proof}

\subsection{Proof of Proposition~\ref{PropR1}}\label{sectProofProp34}
To prove Proposition~\ref{PropR1}, we begin with the following well known bound (see e.g.\ Lemma~4(d), Chapter XIII.17 of~\cite{RS78III})
	\begin{equation}\label{ub13}
	|\det(\Id-\tilde{K})-\det(\Id-K_{1,2}K_{2,1})|\leq \|Q\|_1 e^{\|Q\|_1+2\|K_{1,2}K_{2,1}\|_1+1}.
	\end{equation}
	where $Q=\tilde{K}-K_{1,2}K_{2,1}$, where $\|\cdot\|_1$ is the trace-norm (see e.g.~\cite{Sim00}). From Lemmas~\ref{lem:K12K21} and~\ref{lem:Qb} below, the exponent in the display above is bounded and the result follows.
	\qed

In the remainder of this section we prove the two results Lemma~\ref{lem:K12K21} and Lemma~\ref{lem:Qb} used above. We first need to prove some auxiliary bounds. From the identity
\begin{equation}\label{id}
	(\Id-K_{1,1})^{-1}=\Id+(\Id-K_{1,1})^{-1}K_{1,1}
\end{equation}
we see that
\begin{equation}\label{id2}
	\begin{aligned}
		Q=&\tilde{K} - K_{1,2} K_{2,1}=(\Id-K_{1,1})^{-1}K_{1,1}K_{1,2}(\Id-K_{2,2})^{-1}K_{2,2}K_{2,1}\\
		&+K_{1,2}(\Id-K_{2,2})^{-1}K_{2,2}K_{2,1}+(\Id-K_{1,1})^{-1}K_{1,1}K_{1,2}K_{2,1}.
	\end{aligned}
\end{equation}
Recall that from \eqref{ub13} we need to bound $\|Q\|_1$. Thus it is enough to bound the $\|\cdot\|_1$-norm of each of the terms on the right hand side of \eqref{id2}.
Since $K_{1,2}(x,y)$ is not decaying as $x=y\to \infty$, we could either work in weighted $L^2$ spaces, or as we do here, introduce some weighting in the kernel elements. Namely define
\begin{equation}
	\begin{aligned}
		\bar{K}_{1,2}^L(x,y)&=e^{-\frac x2} K_{1,2}(x,y), \quad
		\bar{K}_{1,2}^R(x,y)= K_{1,2}(x,y)e^{-\frac y2},\\
		\bar{K}_{1,1}(x,y)&= K_{1,1}(x,y) e^{\frac y2},\quad
		\bar{K}_{2,2}(x,y)= e^{\frac x2}K_{2,2}(x,y).
	\end{aligned}
\end{equation}
Also, we use the norm inequalities $\|A B\|_1\leq \|A\|_2 \|B\|_2$ and $\|A\|_2\leq \|A\|_1$, with $\|\cdot\|_2$ the Hilbert-Schmidt norm given by
$\|K\|_2=\Big(\int dxdy (K(x,y))^2\Big)^{1/2}$.

These norm inequalities, the fact that $K_{i,i}$ and $(\Id-K_{i,i})^{-1}$ commute and the identity $K_{1,1}K_{1,2}=\bar{K}_{1,1}\bar{K}_{1,2}^L$ lead to
	\begin{equation}\label{ub17a}
		\|(\Id-K_{1,1})^{-1}K_{1,1}K_{1,2}K_{2,1}\|_1\leq \|(\Id-K_{1,1})^{-1}\|_2\|\bar{K}_{1,1}\|_2\|\bar{K}^L_{1,2}\|_2\|K_{2,1}\|_2.
	\end{equation}	
Moreover, using $K_{1,2}K_{2,2}=\bar{K}^R_{1,2}\bar{K}_{2,2}$, we get
	\begin{equation}\label{ub17b}
		\begin{aligned}
			\|K_{1,2}(\Id-K_{2,2})^{-1}K_{2,2}K_{2,1}\|_1 &=\|K_{1,2}K_{2,2}(\Id-K_{2,2})^{-1}K_{2,1}\|_1 \\
			&\leq \|\bar{K}_{1,2}^R\|_2\|\bar{K}_{2,2}\|_2\|(\Id-K_{2,2})^{-1}\|_2\|K_{2,1}\|_2,
		\end{aligned}
	\end{equation}	
and finally
	\begin{equation}\label{ub17c}
		\begin{aligned}
			& \|(\Id-K_{1,1})^{-1}K_{1,1}K_{1,2}(\Id-K_{2,2})^{-1}K_{2,2}K_{2,1}\|_1 \\
			&\leq \|(\Id-K_{1,1})^{-1}\|_2\|\bar{K}_{1,1}\|_2\|\bar{K}^L_{1,2}\|_2\|(\Id-K_{2,2})^{-1}\|_2\|K_{2,2}\|_2\|K_{2,1}\|_2.
		\end{aligned}
	\end{equation}	
	We turn to bound each of the terms on the right hand side of each of the inequalities in \eqref{ub17a}-\eqref{ub17c}.

We also use the following change of variables: for $a,b\in \R$, define
\begin{equation}\label{LT}
\begin{aligned}
		\omega = (x-a)+(y-b),\quad	\zeta=(x-a)-(y-b),
	\end{aligned}
\end{equation}
which give $x=a+\frac12(\omega+\zeta)$ and $y=b+\frac12(\omega-\zeta)$. So the integral over $(x,y)\in [a,\infty)\times[b,\infty)$ becomes an integral over $(\omega,\zeta)$ with $\omega\geq 0$ and $\zeta\in [-\omega,\omega]$.
\begin{lem}\label{lem1.6}
Uniformly for all $s_1,s_2\geq 0$,
\begin{equation}\label{Ki}
\|K_{i,i}\|_2\leq \tfrac12 e^{-2 s_i},\quad \|\bar{K}_{i,i}\|_2\leq e^{-s_i},\quad \|(\Id-K_{i,i})^{-1}\|_2\leq 2,
\end{equation}
for $i\in\{1,2\}$.
\end{lem}
\begin{proof}
The bounds for $i=1$ and $i=2$ are fully analogous, thus we restrict to the case $i=1$ here. To bound $\|K_{1,1}\|_2^2$ we use \eqref{eqAiExpBound} which gives
	\begin{equation}
		\begin{aligned}
		\|K_{1,1}\|^2_2&=\int_{s_1}^\infty dx \int_{s_1}^\infty dy \big[\Ai(x+y)\big]^2 \leq \int_{s_1}^\infty dx \int_{s_1}^\infty dy e^{-2(x+y)} \leq \tfrac14 e^{-4 s_1}.
		\end{aligned}
	\end{equation}
Thus $\|K_{1,1}\|_2\leq \frac12 e^{-2 s_1}$. Similarly
\begin{equation}
\begin{aligned}
	\|\bar{K}_{1,1}\|^2_2&=\int_{s_1}^\infty dx \int_{s_1}^\infty dy \big[e^{\frac{y}{2}}\Ai(x+y)\big]^2\leq\int_{s_1}^\infty dx \int_{s_1}^\infty dy e^{-(x+2y)}=\tfrac12 e^{-3 s_1},
\end{aligned}
\end{equation}
which implies the second bound in \eqref{Ki}. Finally using
\begin{equation}
	\|\Id-K_{i,i}\|_2^{-1}\leq (1-\|K_{i,i}\|_2)^{-1},
\end{equation}
which holds whenever $\|K_{i,i}\|_2<1$, we get the last inequality.
\end{proof}
	
	\begin{lem}\label{lem1.7}
		For $s_1,s_2\geq 0$ and $u>0$
	\begin{equation}\label{K21}
		\begin{aligned}
	\|K_{2,1}\|_2\leq \frac{1}{2u^{3/2}}\exp\Big[-2(s_1+s_2)u-\tfrac43 u^3\Big].
	\end{aligned}
	\end{equation}
	\end{lem}
\begin{proof}
	Using the change of variables \eqref{LT} with $a=s_2,b=s_1$, it follows that
\begin{equation}\label{eq4}
\begin{aligned}
		\|K_{2,1}\|^2_2&=\int_{s_2}^\infty dx\int_{s_1}^\infty dy\big[\Ai(x+y+u^2)\big]^2 e^{-2(x+y)u-\tfrac43 u^3}\\
		&=2\int_0^\infty d\omega \int_{-\omega}^\omega d\zeta \big[\Ai(\omega+s_1+s_2+u^2)\big]^2 e^{-2(\omega+s_1+s_2)u-\tfrac43 u^3},
\end{aligned}
\end{equation}
where we used that the Jacobian of the transformation in \eqref{LT} is 2.

Next using \eqref{eqAiBetterBound} (with $x=\omega+s_1+s_2$) we get
\begin{equation}\label{ub5}
\begin{aligned}
\eqref{eq4} &= 4e^{-2(s_1+s_2)u-\tfrac43 u^3}\int_0^\infty d\omega \big[\Ai(\omega+s_1+s_2+u^2)\big]^2\omega e^{-2\omega u}\\
&\leq \frac{4e^{-4(s_1+s_2)u-\tfrac83 u^3}}{u}\int_0^\infty d\omega\, \omega e^{-4\omega u}= \frac1{4 u^3}e^{-4(s_1+s_2)u-\tfrac83 u^3}.
\end{aligned}
\end{equation}
	\end{proof}

	\begin{lem}\label{lem1.8}
		For $s_1,s_2\geq 0$, there exists a constant $C_1>0$ such that
		\begin{equation}\label{Krl}
			\begin{aligned}
				\|\bar{K}^R_{1,2}\|_2,\|\bar{K}^L_{1,2}\|_2\leq \frac{C_1}{\sqrt{u}} e^{-s_1/2}
			\end{aligned}
		\end{equation}
for all $u\geq \sqrt{s_1+s_2}$.
	\end{lem}
	\begin{proof}[Proof of Lemma~\ref{lem1.8}]
By symmetry in $x,y$ the bounds for $\|\bar{K}^L_{1,2}\|_2^2$ and for $\|\bar{K}^R_{1,2}\|_2^2$ are the same. So we consider $\|\bar{K}^L_{1,2}\|_2^2$ only. Using $(A-B)^2\leq 2 (A^2+B^2)$ we get
\begin{equation}\label{eq1.68}
\begin{aligned}
\|\bar{K}^L_{1,2}\|_2^2&=\int_{s_1}^\infty dx \int_{s_2}^\infty dy e^{-x}\Big[\Ai(x+y+u^2) e^{(x+y)u+\tfrac23 u^3}-\frac{e^{-(x-y)^2/4u}}{\sqrt{4\pi u}}\Big]^2\\
&\leq 2 \int_{s_1}^\infty dx \int_{s_2}^\infty dy e^{-x}\Big[(\Ai(x+y+u^2))^2 e^{2(x+y)u+\tfrac43 u^3}+\frac{e^{-(x-y)^2/2u}}{4\pi u}\Big].
\end{aligned}
\end{equation}
The second term in \eqref{eq1.68} can be bounded by integrating in $y$ over $\R$ and then integrating over $x\geq s_2$. This gives
\begin{equation}
 2 \int_{s_1}^\infty dx \int_{s_2}^\infty dy e^{-x} \frac{e^{-(x-y)^2/2u}}{4\pi u}\leq \frac{1}{\sqrt{2\pi u}} e^{-s_1}.
\end{equation}
For the first term in \eqref{eq1.68}, we use the change of variable \eqref{LT} with $a=s_1$, $b=s_2$ and obtain
\begin{equation}\label{ub9}
\begin{aligned}
&2 \int_{s_1}^\infty dx \int_{s_2}^\infty dy e^{-x}(\Ai(x+y+u^2))^2 e^{2(x+y)u+\tfrac43 u^3}\\
&=4 e^{-s_1} \int_0^\infty d\omega (\Ai(\omega+s_1+s_2+u^2))^2 e^{2 (\omega+s_1+s_2) u+\tfrac43 u^3}\!\int_{-\omega}^\omega \! d\zeta e^{-\frac12(\omega+\zeta)} \\
&\leq 8 e^{-s_1} \int_0^\infty d\omega (\Ai(\omega+s_1+s_2+u^2))^2 e^{2 (\omega+s_1+s_2) u+\tfrac43 u^3}\\
&\stackrel{\eqref{eqAiExpBound2}}{\leq} 8 e^{-s_1} \int_0^\infty d\omega e^{-\frac43 (\omega+s_1+s_2+u^2)^{3/2}+2 (\omega+s_1+s_2) u+\tfrac43 u^3}.
\end{aligned}
\end{equation}
Let $A(\omega)=-\frac43 (\omega+s_1+s_2+u^2)^{3/2}+2 (\omega+s_1+s_2) u+\tfrac43 u^3$. Then for all $s_1,s_2,u\geq 0$ it is a concave function in $\omega$ and thus
\begin{equation}
\begin{aligned}
A(\omega)&\leq A(0)+A'(0)\omega\\
&=2 (s_1+s_2) u + \frac43 u^3 - \frac43 (s_1+s_2 + u^2)^{3/2}-2(\sqrt{s_1+s_2+u^2}-u)\omega.
\end{aligned}
\end{equation}
The term independent of $\omega$ is always negative. Indeed, $x\mapsto \frac43(x+u^2)^{3/2}$ is convex and thus greater than its linear approximation at $x=0$, which is $\frac43 u^3+2ux$. So we have
\begin{equation}
\begin{aligned}
\eqref{ub9}\leq \frac{4 e^{-s_1} }{\sqrt{s_1+s_2+u^2}-u}\leq \frac{4 e^{-s_1}}{(\sqrt{2}-1)u},
\end{aligned}
\end{equation}
where in the last step we used the assumption that $u\geq \sqrt{s_1+s_2}$.
\end{proof}

	\begin{lem}\label{lem:K12K21}
There exists a constant $C_1>0$ such that for $s_1,s_2\geq 0$,
		\begin{equation}\label{ub16}
			\|K_{1,2}K_{2,1}\|_1\leq \frac{C_1}{2u^{2}} e^{-\frac43 u^3-2(s_1+s_2)u} e^{s_2}\leq \frac{C_1}{2u^{2}} e^{-\frac43 u^3}
		\end{equation}
for all $u\geq \max\{\frac12,\sqrt{s_1+s_2}\}$.
	\end{lem}
\begin{proof}
Denote
	\begin{equation}
		\bar{K}_{2,1}(x,y)=e^{\frac{x}2}K_{2,1}(x,y).
	\end{equation}
Then we have the bound
	\begin{equation}\label{ub15}
		\begin{aligned}
		\|K_{1,2}K_{2,1}\|_1\leq \|\bar{K}_{1,2}^{R}\|_2\|\bar{K}_{2,1}\|_2.
	\end{aligned}
	\end{equation}
From Lemma~\ref{lem1.8}, we get $\|\bar{K}_{1,2}^{R}\|_2\leq\frac{C_1}{u^{1/2}}$. Furthermore, using \eqref{eqAiBetterBound}, we have
\begin{equation}
\begin{aligned}
\|\bar{K}_{2,1}\|_2^2&=\int_{s_2}^\infty dx\int_{s_1}^\infty dy\,e^{x}\big[\Ai(x+y+u^2)\big]^2 e^{-2(x+y)u-\tfrac43 u^3}\\
&\leq \frac{1}{u} \int_{s_2}^\infty dx\int_{s_1}^\infty dy\,e^{x} e^{-4(x+y)u-\tfrac83 u^3}\leq e^{-\frac83 u^3} e^{-4(s_1+s_2)u+s_2}\frac{1}{4u^2(4u-1)}\\
&\leq e^{-\frac83 u^3} e^{-4(s_1+s_2)u+s_2}\frac{1}{8u^3}
\end{aligned}
\end{equation}
for all $u\geq 1/2$.
\end{proof}

We are now ready to bound $\|Q\|_1$.
\begin{lem}\label{lem:Qb}
Uniformly for $s_1,s_2\geq 0$, there exists a constant $C_1$ such that
		\begin{equation}
\|Q\|_1\leq \frac{2 C_1 e^{-\min\{s_1,s_2\}}}{u^2}e^{-2(s_1+s_2)u-\tfrac43 u^3}
		\end{equation}
for all $u\geq \max\{\frac12,\sqrt{s_1+s_2}\}$.
	\end{lem}
\begin{proof}
Applying the bounds of Lemmas~\ref{lem1.6} and~\ref{lem1.8} to \eqref{ub17a}-\eqref{ub17c}, we obtain
\begin{equation}
\|Q\|_1\leq \frac{8 C_1}{u^{1/2}} e^{-\min\{s_1,s_2\}} \|K_{2,1}\|_2
\end{equation}
for all $s_1,s_2\geq 0$ and $u$ as mentioned.
Then, applying the bound of Lemma~\ref{lem1.7} leads to
\begin{equation}
\|Q\|_1\leq \frac{4C_1 e^{-\min\{s_1,s_2\}}}{u^2}e^{-2(s_1+s_2)u-\tfrac43 u^3}.
\end{equation}
\end{proof}

\appendix

\section{Some tail estimates}\label{AppTails}
Here we collect and, in some cases, prove estimates on tails of the Airy function and tails of some LPPs (point-to-point and point-to-half line) that we have used earlier.

\subsection{Bounds on Airy functions}
\label{a:airy}
The Airy function $\Ai$ satisfied the following identity (see e.g.~Appendix~A of~\cite{BFP09})
\begin{equation}\label{ir3}
\frac{1}{2\pi\I}\int_{\gamma}dz \exp\Big(\frac{z^3}3+b z^2-z x\Big)= \Ai(b^2+x)e^{\frac23b^3+b x}
\end{equation}
for any $\gamma=\alpha+\I\R$ with $\alpha>-b$. The asymptotics of $\Ai$ is given by~\cite{AS84}
\begin{equation}\label{asy}
\Ai(x)=\frac{x^{-1/4}}{\sqrt{\pi}} \exp\Big(-\frac23 x^{3/2}\Big)[1+O(x^{-1})] \quad x\rightarrow \infty
\end{equation}
Furthermore, one has the following simple bounds\footnote{$\sup_{x\in\R}|\Ai(x)|\leq c=0.7857\ldots$ follows by $\lim_{n\to\infty} n^{1/3} J_{[2n+u n^{1/3} u]}(2n)=\Ai(u)$ (see also (3.2.23) of~\cite{AS84}) and the bound of Landau~\cite{Lan00}. For any $x\geq 0.01$, the bound \eqref{eqAiExpBound2}, is better that the bound in \eqref{eqAiExpBound} and $e^{-0.01}>c$.}
\begin{equation}\label{eqAiExpBound}
|\Ai(x)|\leq e^{-x},\quad x\in \R
\end{equation}
and, see Equation~9.7.15 of~\cite{NIST:DLMF},
\begin{equation}\label{eqAiExpBound2}
|\Ai(x)|\leq \frac{1}{2\sqrt{\pi}x^{1/4}}e^{-\frac23 x^{3/2}},\quad x\geq 0.
\end{equation}
from which it follows that for all $x,u\geq 0$,
\begin{equation}\label{eqAiBetterBound}
|\Ai(x+u^2)|\leq \frac{1}{\sqrt{u}} e^{-\frac23(x+u^2)^{3/2}}\leq \frac{1}{\sqrt{u}} e^{-\frac23 u^3- x u},
\end{equation}
using that $\frac23(x+u^2)^{3/2}\geq \frac23 u^3+ux$ since $y\mapsto y^{3/2}$ is convex on the positive real line.

\subsection{Bounds on LPP}
\label{a:LPP}
\begin{lem}[Theorem~2 of~\cite{LR10}]\label{lemUB1}
There exist constants $C,c>0$ such that
\begin{equation}\label{eqUB1}
\Pb(L_{(0,0),(N+v(2N)^{2/3},N-v(2N)^{2/3})}\leq 4N-2^{4/3}v^2N^{1/3}-x 2^{4/3} N^{1/3})\leq C e^{-c x^3},
\end{equation}
for all $N^{2/3}\gg x>0$ and $N^{1/3}\gg |v|$.
\end{lem}
Using Riemann-Hilbert methods like as in the case of geometric and Poissonian LPP, see~\cite{BR99b,BR01b,Bai02,BBdF08} for the Riemann-Hilbert problem for exponential LPP, it should be possible to get the optimal constant $c$ as well (expected to be $\frac1{12}$ like for the lower tail of the GUE Tracy-Widom distribution function), but we do not require the optimal constant in the lower tail estimate for our purposes. A simple corollary is the following.

The following corollary follows from the inequality $L_{(0,0),\LL_{2N}}\geq L_{(0,0),(N,N)}$.
\begin{cor}\label{CorUB2}
There exist constants $C,c>0$ such that
\begin{equation}\label{eqUB1Bis}
\Pb(L_{(0,0),\LL_{2N}}\leq 4N-x 2^{4/3} N^{1/3})\leq C e^{-c x^3},
\end{equation}
for all $N^{2/3}\gg x>0$.
\end{cor}

Next we would like to have a sharp upper tail moderate deviation estimate from the last passage time from $(0,0)$ to the half line $\{(N,N)+v(2N)^{2/3}(1,-1):v\ge u\}$.

\begin{lem}\label{lemUB2}
Let ${\cal D}_u=\cup_{v\geq u} J(v)$ with $J(v)=(N,N)+v(2N)^{2/3}(1,-1)$. Then, for $N^{1/9}\gg u>0$ and $N^{2/9}\gg s>0$,
\begin{equation}\label{eqUB2}
\Pb(L_{(0,0),{\cal D}_u}\geq 4N- u^2 2^{4/3} N^{1/3}+ s 2^{4/3} N^{1/3})\leq \frac{C e^{-\frac43 s^{3/2}}}{s \min\{\sqrt{s},u\}}
\end{equation}
for some constant $C$.
\end{lem}
\begin{proof}
Let $\tilde \LL=\{(x,y)\in\Z^2\,|\, x+y=0, y\geq 0\}$. Then for any $a$,
\begin{equation}
\Pb(L_{(0,0),{\cal D}_u}\geq a)=\Pb(L_{\tilde \LL,J(u)}\geq a).
\end{equation}

Next we use the well-known connection between TASEP and LPP (see e.g.~Equation (1.15) of~\cite{BFP09}, which generalizes~\cite{Jo00b}). In our case, the connection with TASEP with half-flat initial condition, i.e., $x_k(0)=-2k$ for $k\geq 0$ are the positions of TASEP particles at time $0$. This gives
\begin{equation}\label{eq1.10}
\begin{aligned}
\Pb(L_{\tilde \LL,J(u)}\geq t)&=\Pb(x_{m}(t)\leq 2u (2N)^{2/3})\\
&=-\sum_{n\geq 1}\frac{(-1)^n}{n!}\sum_{x_1,\ldots,x_n<2u(2N)^{2/3}} \det[K_{m,t}(x_i,x_j)]_{1\leq i,j\leq n},
\end{aligned}
\end{equation}
with $m=N-u (2N)^{2/3}$ and the sum is on $x_1,\ldots,x_n$ being integers below $2u(2N)^{2/3}$; the kernel is given by~\cite{BFS07}
\begin{equation}
K_{m,t}(x,y)=\frac{1}{(2\pi\I)^2}\oint_{\Sigma_0}\frac{dv}{v+1}\oint_{\Sigma_{-1,-1-v}} \hspace{-0.5cm} dw \frac{e^{\Phi(v;m,y)}}{e^{\Phi(w;m,x)}} \left(\frac{1}{w-v}-\frac{1}{1+w+v}\right),
\end{equation}
with $\Phi(v;m,y)=-v t+(y+m)\ln(1+v)-m\ln(v)$.
The contours are simple paths anticlockwise oriented, with $\Sigma_0$ enclosing only the pole at $0$, while $\Sigma_{-1,-1-v}$ enclosing the poles at $-1$ and $-1-v$.
Consider the scaling
\begin{equation}
\begin{aligned}
&x_i=2u(2N)^{2/3}-\xi_i 2^{4/3}N^{1/3},\quad m=N-u(2N)^{2/3},\\
&t=4N-u^2 2^{4/3}N^{1/3}+s 2^{4/3}N^{1/3},\quad K_N(\xi_i,\xi_j)=2^{4/3} N^{1/3} K_{m,t}(x_i,x_j).
\end{aligned}
\end{equation}
Then, with $\delta=2^{-4/3}N^{-1/3}$, we have
\begin{equation}\label{eq1.16}
|\eqref{eq1.10}|\leq \sum_{n\geq 1}\frac{1}{n!} \sum_{\xi_1,\ldots,\xi_n\in \delta \N} \delta^n \left|\det[K_N(\xi_i,\xi_j)]_{1\leq i,j\leq n}\right|.
\end{equation}
The reason why we singled out the $\delta^n$ is because one can think that when $\delta$ is small, the Riemann sum $\sum_{\xi_1,\ldots,\xi_n\in \delta \N} \delta^n f(\xi_1,\ldots,\xi_n)$ is close to the integral $\int_{\R_+^n} d\xi_1\ldots d\xi_n f(\xi_1,\ldots,\xi_n)$. Actually, from the exponential bound we get for $f(\xi_1,\ldots,\xi_n)$ we can bound the estimate with the integral (starting from $-\delta$ instead of $0$ to be precise).

Now we give the choice of the paths for $v$ and $w$ (which are such that they are steep descent). Since it is quite standard, we just indicate some main steps (see e.g.~\cite{BFS07} or Lemma~2.7 of~\cite{CFS16} for a case very similar to the one in this paper). Choose the paths as
\begin{equation}
\begin{aligned}
v&=-(\tfrac12-\e_1)e^{\I \phi},\,\phi\in [-\pi,\pi),\quad \e_1=2^{-4/3}N^{-1/3} u (1+\sqrt{s} u^{-1}),\\
w&=-1+(\tfrac12+\e_2) e^{\I\theta},\,\theta\in [-\pi,\pi),\quad \e_2=2^{-4/3}N^{-1/3} u (1-\sqrt{s} u^{-1}).
\end{aligned}
\end{equation}
On the chosen paths, we have, for $u\ll N^{1/3}$ and $s\ll N^{2/3}$
\begin{equation}
\begin{aligned}
\Re(\Phi(v;m,x_j))&\leq \Phi(-1/2+\e_1;m;x_j) -(1-\cos(\phi)) 2 \sqrt{s}N^{2/3},\\
\Re(-\Phi(w;m,x_i))&\leq -\Phi(-1/2+\e_2;m;x_i) -(1-\cos(\theta)) 2 \sqrt{s}N^{2/3}.
\end{aligned}
\end{equation}
Furthermore
\begin{equation}
\left|\frac{1}{v+1}\right|\leq 2,\quad
\left|\frac{1}{w-v}\right|\leq \frac{2^{1/3}N^{1/3}}{\sqrt{s}},\quad \left|\frac{1}{1+w+v}\right|\leq \frac{2^{2/3}N^{1/3}}{u},
\end{equation}
as the maximal value of the first expression is obtained for the values \mbox{$\theta=\phi=0$}, i.e., when $v$ and $w$ are closest to each other, while $1+w$ and $v$ are two circles around $0$ with radius difference $\e_1-\e_2$.
Putting all together and noticing that $\int_{-\pi}^\pi d\theta e^{-c_1 (1-\cos(\theta))}\leq C/\sqrt{c_1}$ we finally obtain
\begin{equation}\label{eq1.20}
|K_N(\xi_i,\xi_j)|\leq \frac{C}{\sqrt{s}\min\{\sqrt{s},u\}} e^{\Phi(-1/2+\e_1;m;x_j)-\Phi(-1/2+\e_2;m;x_i)}.
\end{equation}
Taylor expansion gives
\begin{equation}\label{eq1.21B}
\begin{aligned}
\Phi(-\tfrac12+\e_1;m;x_j)&-\Phi(-\tfrac12+\e_2;m;x_i) \\
&= g(\xi_j)-g(\xi_i) -\frac43 s^{3/2}(1+E_2)-2\sqrt{s}(\xi_i+\xi_j)(1+E_1)\\
&\leq g(\xi_j)-g(\xi_i) -\frac43 s^{3/2}-\sqrt{s}(\xi_i+\xi_j)+1
\end{aligned}
\end{equation}
for all $N$ large enough and for some explicit conjugation terms $g(\xi)$ (which cancel out exactly when computing the determinant), where $E_1=\Or(s N^{-2/3};u N^{-1/3})=o(1)$ and $s^{3/2}E_2=\Or(s^{5/2} N^{-2/3}; s^{1/2}u^4 N^{-2/3})=o(1)$ do not depend on $\xi_i,\xi_j$.

Inserting \eqref{eq1.21B} and \eqref{eq1.20} in \eqref{eq1.16},using the Hadamard bound and interchanging the order of summation, we get
\begin{equation}
\begin{aligned}
\eqref{eq1.16}&\leq \sum_{n\geq 1}\frac{n^{n/2}}{n!} \frac{C^n e^n}{(\sqrt{s}\min\{\sqrt{s},u\})^n} e^{-\frac{4n}{3} s^{3/2}}\prod_{j=1}^n \sum_{\xi_j\in \delta \N} \delta
 e^{-2\sqrt{s}\xi_j}\\
 &\leq \sum_{n\geq 1}\frac{n^{n/2}}{n!}\left(\frac{C'}{\sqrt{s}\min\{\sqrt{s},u\}} e^{-\frac43 s^{3/2}} \frac{1}{2\sqrt{s}}\right)^n \leq \frac{C''}{s \min\{\sqrt{s},u\}} e^{-\frac43 s^{3/2}},
\end{aligned}
\end{equation}
for some constants $C',C''$. To bound the sum over $\xi_i$, we simply used $\sum_{\xi\in\delta\N}\delta e^{-\alpha \xi}\leq \int_0^\infty e^{-\alpha \xi} d\xi=\alpha^{-1}$, but one could also compute the geometric sum explicitly.
\end{proof}

The next result gives upper tail bounds for point-to-line LPP.

\begin{lem}\label{lemUB2B}
For $N^{2/9}\gg s\geq 0$,
\begin{equation}\label{eqUB2B}
\Pb(L_{(0,0),\LL_{2N}}\geq 4N+ s 2^{4/3} N^{1/3})\leq \frac{C e^{-\frac43 s^{3/2}}}{\max\{1,s^{3/4}\}}
\end{equation}
for some constant $C$.
\end{lem}
\begin{proof}We need only to prove the result for $s\geq 1$ (or any other constant). Then by appropriate choice of the constant $C$, the result holds for all $s\geq 0$ (just take $C$ such that the upper bound is larger than the trivial bound $1$).

We have
\begin{equation}
\Pb(L_{(0,0),\LL_{2N}}\geq t) = \Pb(L_{\LL_{0},(N,N)}\geq t) = \Pb(x_N(t)\leq 0),
\end{equation}
where $x_n(t)$ it the position of TASEP particle $n$ at time $t$ with initial condition $x_n(0)=-2n$, $n\in \Z$. The distribution of TASEP particles are given in terms of a Fredholm determinant
\begin{equation}\label{eqA16}
\Pb(x_N(t)\leq 0) = -\sum_{n\geq 1} \frac{(-1)^n}{n!} \sum_{x_1,\ldots,x_n\leq 0} \det[K_{N,t}(x_i,x_j)]_{1\leq i,j\leq n},
\end{equation}
with the kernel given by~\cite{BFPS06}
\begin{equation}
K_{N,t}(x,y)=\frac{1}{2\pi\I}\oint_{\Sigma_0} \frac{dv}{v} \frac{(1+v)^{y+2N}}{(-v)^{2N+x}} e^{-t(1+2v)}
\end{equation}
with $\Sigma_0$ a simple path anticlockwise oriented enclosing only the pole at $0$.
Setting $x_i=-\xi_i 2^{4/3}N^{1/3}$ and $t=4N+s 2^{4/3}N^{1/3}$ we obtain
\begin{equation}
K_{N,t}(x_i,x_j)=\frac{1}{2\pi\I}\oint_{\Sigma_0} \frac{dv}{v} e^{N f_0(v)+2^{4/3}N^{1/3} f_1(v)}
\end{equation}
with
\begin{equation}
\begin{aligned}
f_0(v)&=-4(1+2v)+2 \ln(1+v)-2 \ln(-v),\\
f_1(v)&=-s(1+2v)-\xi_i \ln(1+v)+\xi_j \ln(-v).
\end{aligned}
\end{equation}
Consider the path parameterized by $v=-\rho e^{\I \theta}$, $\theta\in [-\pi,\pi)$. Then due to
\begin{equation}
\frac{d\Re(f_0(v))}{d\theta}=-2\rho\sin(\theta)\left[4-\frac{1}{|1+v|^2}\right],
\end{equation}
the path is steep descent for any $\rho\in (0,1/2]$. Let us choose the radius by
\begin{equation}
\rho=1/2-\sqrt{s}2^{-4/3}N^{-1/3}.
\end{equation}
For any $s,\xi_i,\xi_j\geq 0$,
\begin{equation}
\Re(2^{4/3} N^{1/3} f_1(v))\leq 2^{4/3}N^{1/3} f_1(-\rho)
\end{equation}
and for $0<s\ll N^{2/3}$,
\begin{equation}
\Re(N f_0(v))\leq N f_0(-\rho)-(1-\cos(\theta)) 4 \sqrt{s} N^{2/3}.
\end{equation}
Finally, integrating over $\theta$ leads to the following estimate on the kernel
\begin{equation}
2^{4/3}N^{1/3} |K_{N,t}(x_i,x_j)|\leq C e^{N f_0(-\rho)+2^{4/3}N^{1/3} f_1(-\rho)}.
\end{equation}
A computation gives
\begin{equation}
\begin{aligned}
&N f_0(-\rho)+2^{4/3}N^{1/3} f_1(-\rho)\\
&= g(\xi_i)-g(\xi_j)-\frac{4}{3} s^{3/2}(1+\Or(s N^{-2/3}))-2\sqrt{s}(\xi_i+\xi_j)(1+\Or(s N^{-2/3}))\\
&\leq g(\xi_i)-g(\xi_j)-\frac{4}{3} s^{3/2}-\sqrt{s}(\xi_i+\xi_j)+1
\end{aligned}
\end{equation}
for some conjugation function $g$, where the error terms do not depend on $\xi_i,\xi_j$. Thus we have the following estimate on the kernel
\begin{equation}
2^{4/3}N^{1/3} |K_{N,t}(x_i,x_j)|\leq \frac{C}{s^{1/4}} e^{g(\xi_i)-g(\xi_j)-\frac{4}{3} s^{3/2}-(\xi_i+\xi_j)\sqrt{s}}
\end{equation}
for some constant $C$.

Plugging this estimate into \eqref{eqA16} and using Hadamard inequality like in the proof of Lemma~\ref{lemUB2} we finally get the claimed bound \eqref{eqUB2B}.
\end{proof}

The final result we need is an upper tail bound for interval-to-line LPP. The proof is similar to that of Lemma 10.6 in \cite{BSS14}, but with the optimal exponent.

\begin{prop}\label{PropUB5} Let ${\cal M}=\{(k,-k)|-\tfrac12(2N)^{2/3}\leq k\leq \tfrac12(2N)^{2/3}\}$ and $\LL_{2N}=\{(N+k,N-k),k\in\Z\}$. Then, for $N^{2/3}\gg s\geq 0$,
\begin{equation}
\Pb\bigg(\sup_{p\in \cal M}L_{p,\LL_{2N}}\geq 4N+s 2^{4/3}N^{1/3}\bigg)\leq C \max\{1,s\} e^{-\frac{4}{3} s^{3/2}},
\end{equation}
for some constant $C>0$.
\end{prop}
\begin{proof}
Notice that it is enough to prove the bound only for $s\geq 1$ (or any other positive constant).

Recall the notation $I(v)=v(2N)^{2/3}(1,-1)$. So ${\cal M}$ is the union of points $I(v)$ with $|v|\leq \tfrac12$. Define the point $w$
\begin{equation}
w=(-\tilde\e N,-\tilde\e N).
\end{equation}

Let us divide $\cal M$ into union of $\tilde\e^{-2/3}$ segments of size $\tilde\e^{2/3}(2N)^{2/3}$, say ${\cal M}=\bigcup_{k=1}^{\tilde\e^{-2/3}} {\cal M}_k$. Then
\begin{equation}
\sup_{p\in {\cal M}}L_{p,\LL_{2N}}=\sup_{1\leq k\leq \tilde\e^{-2/3}} \sup_{p\in {\cal M}_k} L_{p,\LL_{2N}}.
\end{equation}
Using union bound the the fact that each of the $\sup_{p\in {\cal M}_k} L_{p,\LL_{2N}}$ has the same distribution, we get
\begin{equation}\label{eqA31}
\Pb\bigg(\sup_{p\in \cal M}L_{p,\LL_{2N}}\geq S\bigg)\leq \tilde\e^{-2/3} \Pb\bigg(\sup_{|v|\leq \frac12 \tilde\e^{2/3}}L_{I(v),\LL_{2N}}\geq S\bigg).
\end{equation}
Next we want to bound the last probability in \eqref{eqA31}.

The idea is the following. For fixed $v$, we know that the same tail estimates we want to prove hold for $L_{I(v),\LL_{2N}}$. We also have an estimate of the upper tail for $L_{w,\LL_{2N}}$ and that with positive probability, $L_{w,I(v)}$ can not be too small in the $N^{1/3}$ scale. So, the fluctuations coming from maximizing $L_{I(v),\LL_{2N}}$ over $|v|\leq \tfrac12(1-\tilde\e)^{2/3}$ can not be compensated fully from the fluctuations of $L_{w,I(v)}$. This will imply our claim.

Let $\hat L_{w,I(u)}=L_{w,I(u)}-\omega_{I(u)}$ be the LPP from $w$ to $I(u)$ without the random variable at the end-point $I(u)$~\footnote{Removing the end-point does not influence the asymptotics and bounds, but it has the property that $\hat L_{w,I(u)}$ and $L_{I(u),\LL_{2N}}$ are independent random variables and the concatenation property holds true.}.
We know from~\cite{FO17,Pim17} that
\begin{equation}
v\mapsto \frac{\hat L_{w,I(v \tilde\e^{2/3})}-4\tilde\e N}{2^{4/3}(\tilde\e N)^{1/3}}
\end{equation}
is tight in the space of continuous functions on compact sets. As a consequence there exists a constant $C'>0$ such that the event
\begin{equation}
{\cal H}=\Big\{\inf_{|v|\leq \frac12 \tilde\e^{2/3}} \hat L_{w,I(v)}\geq 4\tilde\e N- C' 2^{4/3}(\tilde\e N)^{1/3}\Big\}
\end{equation}
satisfies $\Pb({\cal H})\geq 1/2$.

We also have
\begin{equation}
L_{w,\LL_{2N}}\geq \sup_{|v|\leq \frac12 \tilde\e^{2/3}}(\hat L_{w,I(v)}+L_{I(v),\LL_{2N}})\geq \inf_{|v|\leq \frac12 \tilde\e^{2/3}} \hat L_{w,I(v)}+\sup_{|v|\leq \frac12 \tilde\e^{2/3}} L_{I(v),\LL_{2N}}.
\end{equation}
Define the event
\begin{equation}
{\cal G}=\Big\{\sup_{|v|\leq \frac12 \tilde\e^{2/3}}L_{I(v),\LL_{2N}}\geq 4N+s 2^{4/3}N^{1/3}\Big\}.
\end{equation}
Then,
\begin{equation}
\Pb(L_{w,\LL_{2N}}\geq 4(1+\tilde\e)N+(s-C' \tilde\e^{1/3}) N^{1/3})\geq \Pb({\cal H}\cap {\cal G})=\Pb({\cal H})\Pb({\cal G}),
\end{equation}
where we used independence of $\cal H$ and $\cal G$.
Thus we have shown that
\begin{equation}
\Pb({\cal G})\leq 2 \Pb(L_{w,\LL_{2N}}\geq 4(1+\tilde\e)N+(s-C'\tilde\e^{1/3}) N^{1/3}) \leq 2 C e^{-\frac43\frac{(s-C'\tilde\e^{1/3})^{3/2}}{(1+\tilde\e)^{1/2}}}
\end{equation}
where the last inequality holds for $N^{2/3}\gg s-C'\tilde\e^{1/3}>1$ by Lemma~\ref{lemUB2B}. Replacing this into \eqref{eqA31} we get
\begin{equation}
\Pb\bigg(\sup_{p\in \cal M}L_{p,\LL_{2N}}\geq 4N+s 2^{4/3}N^{1/3}\bigg)\leq 2 C \tilde\e^{-2/3} e^{-\frac43\frac{(s-C'\tilde\e^{1/3})^{3/2}}{(1+\tilde\e)^{1/2}}}.
\end{equation}
Consider first $s\geq 1$. Then taking $\tilde\e=1/s^{3/2}$ (notice that $\tilde\e N \gg 1$ since we assumed $s\ll N^{2/3}$), we get
\begin{equation}
\Pb\bigg(\sup_{p\in \cal M}L_{p,\LL_{2N}}\geq 4N+s 2^{4/3}N^{1/3}\bigg) \leq C s e^{-\frac43 s^{3/2}}
\end{equation}
for some new constant $C>0$.
\end{proof}


\end{document}